\documentclass[a4paper,12pt]{article}
\usepackage[utf8]{inputenc}

\usepackage{graphicx}
\usepackage{newtxtext}
\usepackage[a4paper,width=17cm, height=26cm]{geometry}
\usepackage{pinlabel}
\usepackage{amsmath}
\usepackage{amsfonts}
\usepackage{amssymb}
\usepackage{enumerate}
\usepackage{tikz}
\usetikzlibrary{decorations.pathreplacing}
\usetikzlibrary{calc}
\usetikzlibrary{shapes.geometric}
\usetikzlibrary{arrows}
\usepackage{empheq}
\usepackage{tkz-euclide}
\usetkzobj{all}
\usepackage{pgfplots}
\pgfplotsset{compat=1.14}
\usepgfplotslibrary{ternary}

\usepackage{yhmath}

\usepackage{amsthm}
\usepackage{dsfont}
\usepackage{floatflt}
\usepackage{mathrsfs}
\usepackage{url}
\newcommand{\R}{\mathbb R}
\newcommand{\C}{\mathbb C}

\newcommand{\Sp}{{\mathbb S}}

\newcommand{\hC}{\hat{\C}}
\newcommand{\area}{\mathrm{Area}}
\newcommand{\Const}{\mathsf{Const}}
\newcommand{\const}{\mathsf{const}}

\theoremstyle{plain}
\newtheorem{theorem}{Theorem}[section]

\newtheorem{corollary}[theorem]{Corollary}
\newtheorem{lemma}[theorem]{Lemma}
\theoremstyle{definition}
\newtheorem{remark}[theorem]{Remark}

\theoremstyle{definition}

\title{Convergence of discrete period matrices and discrete holomorphic 
integrals for ramified coverings of the Riemann sphere}
\author{Alexander I.\ Bobenko and Ulrike B\"ucking}

\begin{document}

\maketitle

\begin{abstract}
We consider the class of compact Riemann surfaces which 
are ramified coverings of the Riemann sphere $\hC$. Based on a triangulation of 
this covering of the sphere $\Sp^2\cong \hC$ and its stereographic projection, 
we define discrete (multi-valued) harmonic and holomorphic functions. 
We prove that the corresponding discrete period matrices converge to their 
continuous counterparts. In order to achieve an error estimate, which is linear 
in the maximal edge length of the triangles, we suitably adapt the 
triangulations in a neighborhood of every branch point. 
Finally, we also prove a convergence result for discrete holomorphic integrals 
for our adapted triangulations of the ramified covering.
\end{abstract}

\section{Introduction}

Smooth holomorphic functions can be characterized in different ways. In 
particular, the real and imaginary 
part of any holomorphic function is harmonic and both are related by the 
Cauchy-Riemann equations. This perspective
naturally led to linear discretizations of harmonic and holomorphic functions, 
starting with results for square grids, 
see~\cite{CFL28,Is41,Du53}. Lelong-Ferand further developed this theory of 
discrete harmonic and holomorphic functions in~\cite{F,LF55}. MacNeal and Duffin
generalized these notions in~\cite{MacN49,Du56,Du59,Du68}. In particular, 
they considered arbitrary triangulations in the plane and discovered the 
cotan-weights. The cotan-Laplacian is also considered for triangle 
meshes, for example for surfaces in discrete differential 
geometry, see~\cite{PP93}, or for applications in computer 
graphics, see for example~\cite{MDSB03}.
Further properties and theorems of the smooth
theory of holomorphic functions have found recently discrete analogues in the 
discrete linear theory in~\cite{BG16,BG17}.

Note that there are other important nonlinear  discretizations of holomorphic 
functions, for example involving circle packings or circle 
patterns~\cite{St05,Sch97,BS02,Bue08}, connected to 
cross-ratios~\cite{BP96,Ma05}, using discrete conformal 
equivalence~\cite{BPS13,Bue16}, or based on bi-colored 
triangles~\cite{DN03,N11}. The linear theory of holomorphic functions on rhombic 
lattices can be obtained as infinitesimal deformation of circle 
patterns~\cite{BMS05}.

Mercat generalized in~\cite{Me01} the discrete linear theory from planar 
subsets to discrete 
Riemann surfaces and introduced in~\cite{Me02,Me07} discrete period matrices. 
In~\cite{BMS11} numerical experiments are considered
to compute discrete period matrices for polyhedral surfaces explicitly and 
compare them to known period matrices for the corresponding smooth surfaces. 
A convergence proof for the class of polyhedral surfaces was obtained 
in~\cite{BoSk16}. 

The interest in numerical computation of period matrices is for example 
motivated by the computation of finite-genus solutions of integrable 
differential equations.
As Riemann surfaces may be represented as algebraic curves, this is often taken 
as a starting point for computing discrete period matrices. Recent results in 
this context include~\cite{GSST98,DH01,FK15,FK17,MN17}. 

In this article, we take a different approach and consider Riemann  surfaces 
which 
are ramified coverings of the Riemann sphere $\hC$. Based on a triangulation of 
this covering of the sphere $\Sp^2\cong \hC$ or its stereographic projection, 
discrete period matrices can be obtained from this discrete data. Furthermore, 
we prove convergence of the discrete period matrices to their continuous 
counterparts (Theorem~\ref{theoPeriodConv}). In particular, we obtain an error 
estimate, which is linear in the maximal edge length of the 
triangles if we adapt the triangulations in a neighborhood of every branch 
point. The details of our `adapted triangulations' will be  
explained in Section~\ref{subsecConv1}. 

The convergence of discrete analytic functions to their continuous  counterparts 
remains an important issue, although several results have been proved by now. In 
particular, for the linear theory, convergence was first shown for the square 
lattice~\cite{CFL28,LF55} and recently for more general quadrilateral 
lattices~\cite{ChS11,Sko13,BoSk16}. In this article, we prove the convergence of 
discrete holomorphic integrals (Abelian integrals of first kind) obtained 
from suitable triangulations 
of the ramified covering to their continuous counterparts 
(Theorem~\ref{theoConvAbelInt}). 

Our main results are stated in Section~\ref{SecResults} and proved  in 
Section~\ref{SecConvProof}. The proof is inspired by~\cite{BoSk16} and uses 
energy estimates which allow to prove the convergence of the discrete period 
matrices directly. Our results are also applied to improve the convergence 
results of~\cite{BoSk16} in Section~\ref{SecPoly}. Finally, in 
Section~\ref{secNum}, we present some numerical experiments.

\section{Convergence results for discrete period matrices and discrete 
holomorphic integrals for ramified coverings of $\hC$}\label{SecResults}


In the following, we consider any compact Riemann surface ${\cal R}$ of genus 
$g\geq 1$ which allows a branched covering map $f:{\cal R}\to\hC$. 
Using this covering map as a local chart, we always locally identify points in 
${\cal R}$ with their images in $\hC$.
%
Then for points in $\hC\setminus\{\infty\}$ we apply the standard stereographic 
projection to the complex plane $\C$. 
This map from $\cal R$ to $\C$ is denoted by $Pr_{\cal R}$ and gives a local 
chart in a neighborhood about every point, except at branch points.
For further use, we fix a radius $\varrho>1$ such 
that the images of all branch points, except possibly $\infty$, have a distance 
at most $\varrho/2$ from the origin.

Let $T=T_{\cal R}$ be a triangulation of ${\cal R}$ such that all 
branch points are vertices. We assume that every triangle is contained in 
only one sheet of the covering. We will mostly consider this triangulation via 
its (local) image under the chart $Pr_{\cal R}$. In this sense, without further 
mention, we always {identify} this 
triangulation with the corresponding (multi-sheeted) triangulation on $\hC$ 
(which is the image 
$f(T)$ under the covering map) and with the (multi-sheeted) image of this 
triangulation of $\C$ under the map $Pr_{\cal R}$, excluding the vertex 
at infinity. We 
assume that this triangulation is 
a locally planar embedding in the complex plane $\C$ or equivalently in the 
Riemann sphere $\hC$, except at the branch points and possibly at $\infty$. 
From now 
on, we consider the vertices of the triangulation as points of $\C$, except 
$\infty$, that is, we always apply $Pr_{\cal R}$, i.e.\ the covering map and 
the standard 
stereographic projection. The edges connecting incident vertices will be 
straight 
line segments or circular arcs in $\C$, depending on the following distinction.
\begin{enumerate}[(i)]
 \item All triangles with at least two vertices in the 
open disk $B_{\varrho}(0)$ of radius $\varrho$ about the origin are 
geodesic, that is Euclidean triangles. We always consider these triangles to 
be embedded in $\C$.
 \item All triangles whose vertices are 
all contained in the complement $\C\setminus B_{\varrho}(0)$ are preimages of 
a geodesic triangulation with Euclidean triangles under the map $z\mapsto 
1/z$. Therefore, these triangles are in general bounded by circular arcs. 

Note that we mostly consider the images of these triangles under the map 
$z\mapsto 1/z$ which are Euclidean triangles embedded into $\C$ in a 
neighborhood of the origin.
\item The remaining triangles in the 'boundary region' are consequently in 
general bounded 
by two straight lines and one circular arc. These triangles will be called {\em 
boundary triangles} and denoted by $F_{\varrho}$. Finally, we assume that the 
edge lengths of all boundary triangles are strictly smaller than 
$\max\{\varrho/2,1\}$.
As in the first case, we always consider these triangles to be embedded in 
$\C$. 
\end{enumerate}


We denote by $V,E,\vec{E},F$ the sets of vertices, edges, oriented edges, and 
faces of $T_{\cal R}$, respectively, and identify them locally with their 
images under the map  $Pr_{\cal R}$.

\subsection{Discrete harmonic functions}\label{secDiscHarm}

We define weights on the edges $E$ of the triangulation $T_{\cal R}$ 
essentially by using cotan-weights, but we distinguish three cases for 
edges $e=[x,y]\in E$ corresponding to the different cases above:
\begin{enumerate}[(i)]
 \item If both triangles incident to $e$ are contained in the open disk 
$B_{\varrho}(0)$, we use cotan-weights 
\begin{equation}\label{eqcotandef}
 c(e)=\frac{1}{2}\cot \alpha_e +\frac{1}{2}\cot \beta_e,
\end{equation}
where $\alpha_e$ and $\beta_e$ are the angles opposite to the edge $e\in E$ in 
the two adjacent triangles, see Figure~\ref{figCot}.

\begin{figure}[t]
\begin{center}
\vspace{-5em}
\begin{tikzpicture}[scale=1.5]
  \coordinate  (B) at (2cm,0);
  \coordinate  (C) at (0cm, 1cm);
\coordinate (D) at (2cm, 1.732cm);
\coordinate  (E) at (5cm,0.5cm);
\tkzLabelPoint[below](B){$t_e$};
\tkzLabelPoint[above](D){$h_e$};
 
\draw (B) -- (C)  -- (D) -- (E) -- (B); 
\draw [thick, decoration={markings,mark=at position 1 with
    {\arrow[scale=3,>=stealth]{>}}},postaction={decorate}] (B) -- (D);

\tkzCircumCenter(D,B,C)\tkzGetPoint{G1};
\tkzDrawPoint(G1);
\tkzLabelPoint[right](G1){$l_e$};
  \draw [dotted] (G1) -- (B) ;
\draw  [dotted] (G1) -- (C) ;
\draw [dotted] (G1) -- (D) ;

\tkzCircumCenter(D,B,E)\tkzGetPoint{G2};
\tkzDrawPoint(G2);
\tkzLabelPoint[left](G2){$r_e$};
\draw [dotted] (G2) -- (E) ;
\draw [dotted] (G2) -- (B) ;
\draw [dotted] (G2) -- (D) ;

\tkzMarkAngle[size=0.7, fill=green, opacity=0.5](B,C,D);
\tkzLabelAngle[pos=0.5](B,C,D){$\alpha_e$};

\tkzMarkAngle[size=0.7, fill=green, opacity=0.5](D,E,B);
\tkzLabelAngle[pos=-0.5](B,E,D){$\beta_e$};
\end{tikzpicture}
\vspace{-5em}
\end{center}
\caption{Notation associated with an edge $e=[t_e,h_e]\in E$ and with its 
oriented version $\vec{e}= \protect\overrightarrow{t_eh_e}$.
}\label{figCot}
\end{figure}
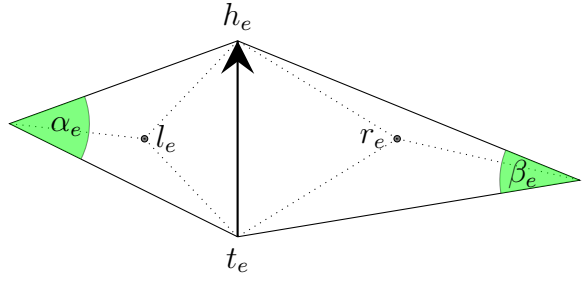

\item If both triangles incident to $e$ are contained in $\hC\setminus 
B_{\varrho}(0)$, we consider the image of the triangulation under the map 
$z\mapsto 1/z$. We define the edge weights by~\eqref{eqcotandef} for the {\em 
image} triangles. Note that this amounts to using in~\eqref{eqcotandef} the 
angles opposite to the 
edge $e\in E$ in the two adjacent circular arc triangles.
\item If $e=[x,y]$ is incident to a boundary triangle in $F_{\varrho}$, we 
define the weight similarly as above as a sum $c(e)=C_1+C_2$ of two parts 
corresponding to the two incident triangles $\Delta_1,\Delta_2$.
If there is a non-boundary triangle, say $\Delta_1$, incident to $e$, we 
consider the angle $\alpha_e$ in this triangle opposite to $e$ and set 
$C_1=\frac{1}{2}\cot\alpha_e$.
The second part $C_2=C_{[x,y]}$ is defined below in~\eqref{eqdefc} using a 
suitable interpolation function and the 
smooth Dirichlet energy. More details and explicit 
calculations are given in Appendix~\ref{secApp1}.
\end{enumerate}

Using our edge weights, we can define discrete harmonicity and a discrete 
Dirichlet energy for functions $u:V\to\R$ on the vertices of the triangulation 
$T_{\cal R}$. In particular, $u$ is called {\em discrete harmonic} if 
for every vertex $x\in V$ there holds
\begin{equation}\label{eqdefharm}
 \sum\limits_{y\in V: [x,y]\in E} c([x,y])(u(y)-u(x))=0.
\end{equation}
The {\em energy of $u$} is
\begin{equation}\label{eqdefE}
 E_T(u)=\sum\limits_{e=[x,y]\in E} c([x,y])(u(y)-u(x))^2.
\end{equation}

The motivation for our choice of weights, in particular for the choice of 
weights for boundary triangles, is the following connection of 
discrete and smooth Dirichlet energies. Recall that for a continuous function 
on a compact Riemann surface $\cal R$ which is smooth almost everywhere the 
{\em Dirichlet energy} is defined as
\[ E(u)=\int_{\cal R} |\nabla u|^2.\]
Then the discrete energy of a function $u:V\to\R$ is in fact the Dirichlet 
energy of the continuous {\em interpolation function} $I_Tu$, defined piecewise 
on every triangle $\Delta[x,y,z]$ as follows:
\begin{enumerate}[(i)]
 \item If at least two of the three vertices $x,y,z$ are contained 
in the open disk 
$B_{\varrho}(0)$, we define $I_Tu|_{\Delta[x,y,z]}$ on this triangle as the 
linear 
interpolation of the values of $u$ at the vertices.
\item If all vertices $x,y,z$ are contained in $\hC\setminus 
B_{\varrho}(0)$, we consider the image of the triangle under the map $z\mapsto 
1/z$. This is a Euclidean triangle $\widetilde{\Delta}$. Let 
$\widetilde{u}=u\circ 1/z$ 
be the corresponding transformed function and denote by 
$\widetilde{I_Tu}_{\widetilde{\Delta}}$ the linear interpolation of 
$\widetilde{u}$ on 
$\widetilde{\Delta}$. We define $I_Tu$ on the original triangle as the 
corresponding value of $\widetilde{I_Tu}_{\widetilde{\Delta}}$, so 
$I_Tu|_{\Delta[x,y,z]}= \widetilde{I_Tu}_{\widetilde{\Delta}}\circ 1/z$
\item In the remaining case, $\Delta[x,y,z]$ is a boundary triangle in 
$F_{\varrho}$ and there is exactly one vertex in $B_{\varrho}(0)$, say 
$x$. We first define $I_Tu$ on the boundary edges consistently with the 
definitions in~(i) and~(ii). The two edges $[x,y]$ and $[x,z]$ are straight 
lines. On these edges 
we define $I_Tu$ as the linear interpolation of the values of the vertices. On 
the arc $\wideparen{yz}$ connecting $y$ and $z$ we
use the interpolating function from~(ii). Then for every straight line segment
connecting $x$ to a point on the arc 
$\wideparen{yz}$ we define $I_Tu$ as the linear interpolation of the values on 
the endpoints.

Using this interpolation function, we obtain by elementary 
calculations (see Appendix~\ref{secApp1} for details) that 
\begin{equation}\label{eqdefc}
\int_{\Delta[x,y,z]} |\nabla I_Tu|^2= C_{[x,y]}(u(x)-u(y))^2 + 
C_{[y,z]}(u(y)-u(z))^2 +C_{[z,x]}(u(z)-u(x))^2,
\end{equation}
where the constants $C_{[x,y]},C_{[y,z]},C_{[z,x]}$ only depend on the triangle
$\Delta[x,y,z]$ (see~\eqref{eqCxy}--\eqref{eqCzx}) and give one part of the 
weights associated to the edges $[x,y],[y,z],[z,x]$ respectively.
\end{enumerate}

It is easy to see that for every triangulation of a ramified covering ${\cal 
R}$ of $\hC$ as above, $I_Tu$ is a well-defined continuous function on ${\cal 
R}$. Furthermore, we have

\begin{lemma}[Interpolation lemma]\label{lemInterpolation}
\[E(I_Tu)=E_T(u).\]
\end{lemma}
\begin{proof}
 We can split the energy according to triangles $\Delta_f$ for $f\in F$.
\begin{align*}
 E(I_Tu) &=\sum\limits_{\Delta_f\subset B_{\varrho}(0)} \int_{\Delta_f} |\nabla 
I_Tu|^2
+\sum\limits_{\Delta_f\subset  \hC\setminus B_{\varrho}(0)} 
\int_{\Delta_f} |\nabla I_Tu|^2
+\sum\limits_{\Delta_f\in F_{\varrho}(0)} \int_{\Delta_f} |\nabla I_Tu|^2
\end{align*}
In particular, elementary calculations show that for any triangle 
$\Delta[x,y,z]$ we have 
\begin{equation}
\int_{\Delta_f} |\nabla I_Tu|^2= C_{[x,y]}(u(x)-u(y))^2 + 
C_{[y,z]}(u(y)-u(z))^2 +C_{[z,x]}(u(z)-u(x))^2,
\end{equation}
where the constants $C_{[x,y]},C_{[y,z]},C_{[z,x]}$ only depend on the triangle 
$\Delta[x,y,z]$. 
Duffin showed in~\cite[\S~4]{Du59} that for Euclidean 
triangles these constants are one half of the cotan of the opposite angles. 
Using the conformal invariance of the Dirichlet integral 
for the triangles in $\hC\setminus B_{\varrho}(0)$
and our choice of weights from~\eqref{eqdefc}, we obtain the claim.
\end{proof}

\begin{remark}\label{remest}
It will be important to note that the constants $C_{[x,y]},C_{[y,z]},C_{[z,x]}$ 
defined by~\eqref{eqdefc} are only small perturbations of the usual 
cotan-weights in 
the following sense. If the maximal edge length $h$ in the boundary  
triangle $\Delta[x,y,z]\in F_\varrho$ is small enough and 
the angles in $\Delta[x,y,z]$ as well as the angles in the Euclidean triangle 
formed by the vertices $x,y,z$ lie in $[\delta,\pi-\delta]$ for some 
$\delta>0$, then elementary calculations and estimates show that
\begin{equation*}
 \textstyle\frac{1}{2}\cot\hat{\alpha}_e -C_{\delta,\varrho}\cdot h \leq C_e 
\leq \textstyle\frac{1}{2}\cot\hat{\alpha}_e +C_{\delta,\varrho}\cdot h.
\end{equation*}
for some constant $C_{\delta,\varrho}>0$, where $\hat{\alpha}_e$ denotes the 
angle opposite to the edge $e$ in 
the Euclidean triangle formed by the vertices $x,y,z$. 
The details are given in Appendix~\ref{secApp1}.

Note that the difference between the the angles in the Euclidean 
triangle with vertices $x,y,z$ and the corresponding angles in $\Delta[x,y,z]$ 
is of order $h$. Thus, for $h$ small enough, the corresponding estimates on 
$|C_{e}-\frac{1}{2}\cot\alpha_e|$ also hold for the actual angles $\alpha_e$ in 
the triangle $\Delta[x,y,z]$.
\end{remark}

\subsection{Discrete analytic functions, discrete holomorphic integrals and 
discrete period matrices}

In the following, we define discrete analytic functions and discrete holomorphic
integrals analogously as in~\cite{BoSk16}.

For an oriented edge $\vec{e}\in\vec{E}$, we denote by $h_e\in V$ and 
$t_e\in V$ the {\em head} and the {\em tail} of $\vec{e}$, and by 
$l_e\in F$ and $r_e\in F$ the {\em left shore} and the {\em right shore} of 
$\vec{e}$, respectively, see Figure~\ref{figCot}. Two functions $u:V\to\R$ and 
$v:F\to \R$ are {\em conjugate}, if for each oriented edge $\vec{e}\in\vec{E}$ 
we have
\begin{equation}\label{eqdefconj}
 v(l_e)-v(r_e)=c(e)(u(h_e)-u(t_e)).
\end{equation}
The pair $f=(u:V\to\R,\; v:F\to\R)$ of two conjugate functions is called a {\em 
discrete analytic function}.
We write $\text{Re} f:=u$ and $\text{Im} f :=v$. If both $u$ and $v$ are 
constant functions, not necessarily equal to each other, we write 
$f=\text{const}$. A direct checking shows that on simply connected surfaces
discrete harmonic functions are precisely real parts of discrete analytic 
functions. Note that for non-zero weights $c(e)\not=0$ we define the (discrete) 
{\em energy} of a function $v:F\to \R$ by 
\begin{equation}\label{eqdefEconj}
E_T(v):=\sum\limits_{e=[x,y]\in E} \frac{1}{c([x,y])}(v(y)-v(x))^2.
\end{equation}

We will consider multi-valued functions on the vertices and the faces of the 
triangulation $T$. Informally, a multi-values function changes its values after 
performing some nontrivial loop on the surface.

Recall that the Riemann surface ${\cal R}$ is a 
branched covering of $\hC$ with genus $g\geq 1$. Denote by 
$p:\widetilde{{\cal R}}\to {\cal R}$ the universal covering of $\cal R$ and by 
$p:\widetilde{T}\to T$ the induced universal covering of $T$. Fix a base point 
$z_0\in\widetilde{{\cal R}}$ and 
closed paths $\alpha_1,\dots,\alpha_g,\beta_,\dots,\beta_g:[0,1]\to \cal R$ 
forming a standard basis of the fundamental group $\pi_1({\cal R},p(z_0))$ such 
that $\alpha_1\beta_1\alpha_1^{-1}\beta_1^{-1}\cdots 
\alpha_g\beta_g\alpha_g^{-1}\beta_g^{-1}$ is null-homotopic.
Each closed path $\gamma:[0,1]\to \cal R$ with $\gamma(0)=\gamma(1)=p(z_0)$ 
determines the {\em deck transformation} $d_\gamma: \widetilde{{\cal 
R}}\to\widetilde{{\cal R}}$, 
that is, the homeomorphism such that $p\circ d_\gamma=p$ and 
$d_\gamma(z_0)=\tilde{\gamma}(1)$, where $\tilde{\gamma}:{\cal 
R}\to\widetilde{{\cal R}}$ is 
the lift of $\gamma:[0,1]\to S$ such that $\tilde{\gamma}(0)=z_0$. The 
induced
deck transformation of $\widetilde{T}$ is also denoted by $d_\gamma: 
\widetilde{T}\to\widetilde{T}$.

A {\em multi-valued function with periods $A_1,\dots, A_g, B_1,\dots, 
B_g\in\C$} is a pair of functions $f=(\text{Re} f :\widetilde{V}\to\R, \;
\text{Im} f:\widetilde{F}\to\R)$ such that for each $k=1,\dots,g$ an each $z\in 
\widetilde{V}$, $w\in\widetilde{F}$ we have
\begin{align*}
 \text{Re} f(d_{\alpha_k}(z))-\text{Re} f(z)&=\text{Re}(A_k), 
& \text{Re} f (d_{\beta_k}(z))-\text{Re} f(z)&=\text{Re}(B_k), \\
 \text{Im} f(d_{\alpha_k}(z))-\text{Im} f(z)&=\text{Im}(A_k), 
& \text{Im} f (d_{\beta_k}(z))-\text{Im} f(z)&=\text{Im}(B_k).
\end{align*}
The numbers $A_1,\dots,A_g$ and $B_1,\dots,B_g$ are called the {\em 
$A$-periods} and the {\em $B$-periods} of the multi-valued function $f$, 
respectively. Analogously, we can also define multi-valued functions 
$u:\widetilde{V}\to\R$, $v:\widetilde{F}\to\R$ or $u:\widetilde{{\cal 
R}}\to\R$. Note in 
particular, that for each multi-valued function $u:\widetilde{V}\to\R$ and 
every edge $[x,y]\in E$ the difference $u(x)-u(y)$ is well defined. The 
(discrete) {\em energy} of the multi-valued function is
\[ E_T(u)=\sum\limits_{[x,y]\in E} c([x,y])(u(x)-u(y))^2.\]
Similarly, for each multi-valued function $u:\widetilde{{\cal R}}\to\R$, which 
is smooth on every face of $\widetilde{F}$, at each 
point inside a face $\Delta\in F$ the gradient $\nabla u$ is well defined. The
(Dirichlet) {\em energy} of the multi-valued function is
\[ E(u)=\sum\limits_{\Delta_f\in F} \int_{\Delta_f} |\nabla u|^2.\]

A multi-valued discrete analytic function is called a {\em discrete holomorphic 
integral} or {\em discrete Abelian integral of the first kind}.


\begin{theorem}[{\cite[Theorem~2.3]{BoSk16}}]
 For any numbers $A_1,\dots,A_g\in\C$ there exists a discrete holomorphic 
integral with $A$-periods $A_1,\dots,A_g$. It is unique up to a 
constant.
\end{theorem}

For each $l=1,\dots,g$ denote by $\phi^l_T=(\text{Re} \phi^l_T 
:\widetilde{V}\to\R, \;
\text{Im} \phi^l_T:\widetilde{F}\to\R)$ the unique (up to constant) discrete 
holomorphic integral with $A$-periods given by 
$A_k=\delta_{kl}$, 
where $k=1,\dots,g$. The $g\times g$-matrix $\Pi_T$ whose $l$-th column is
formed by the $B$-periods of $\phi^l_T$, where $l=1,\dots,g$, is called the 
{\em period matrix} of the triangulation $T$.

\subsection{Convergence of energy and discrete period matrices}\label{subsecConv1}

So far, we have defined our notions like discrete energy for a rather general 
class of triangulations. In view of our convergence results, we now make some 
additional assumptions. In orde to measure distances and other 
metric properties, we always consider the images of the triangles in $\C$ under 
the 
projection $Pr_{\cal R}$. By our assumptions above, these are Euclidean 
triangles if they are contained in $B_\varrho(0)$ and we use the  
standard metric in $\C$. We also apply this metric for boundary triangles in 
$F_\varrho$. For triangles which are mapped to $\hC\setminus B_\varrho(0)$, we 
consider their image under the map $z\mapsto 1/z$ and then use again the 
standard metric. Alternatively, we could work on $\hC$ with the 
chordal metric.

First we determine the 
maximal distance between two vertices in a triangle which lies inside 
$B_{\varrho}(0)$ or on the boundary $F_{\varrho}$ and the maximal edge 
length of the triangles outside $B_{\varrho}(0)$ after the mapping by 
$1/z$. The maximum of these two numbers is called {\em maximal edge 
length} and denoted by $h=h(T)$. 

Furthermore, we suppose that near all branch points of
$\cal R$ the edge lengths are adapted to the singularity which then guarantees 
an approximation error of order $h$. 
In particular, for every branch point $v_O\in V$ with $f(v_O)\not=\infty$, 
choose a radius $r_O$ such that the disks of these radii are 
disjoint for different points $O=Pr_{\cal R}(v_O)\in\C$. Furthermore,  we assume 
that 
all these disks are all contained in $B_{\varrho}(0)$. 
Let ${\cal C}_O^{\cal R}$ be the 
neighborhood of $v_O$ in $\cal R$ which projects onto this disc 
$B_{r_O}(O)=Pr_{\cal R}({\cal C}_O^{\cal R})$.
If $O=\infty$, we first apply the mapping $z\mapsto 1/z$ and assume that 
$r_O=1/(2\varrho)$.
For all branch points we already have a natural complex structure and 
charts. In particular if $O\not=\infty$, we consider the chart 
$g_O(z)=(z-O)^{\gamma_O}$, so $g_O\circ Pr_{\cal R}$
maps  ${\cal C}_O^{\cal R}$ onto a neighborhood of the origin in $\C$. If 
$O=\infty$, consider $g_O(z)=1/z^{\gamma_O}$ instead.
We can also introduce ``polar coordinates'' 
$(r,\phi)$ on ${\cal C}_O^{\cal R}$ with the origin at the vertex $O$. We map 
all vertices of $T$ in ${\cal C}_O^{\cal R}$ to a neighborhood of the origin in 
$\C$ by the chart 
$g_O(r,\phi):=r^{\gamma_O}\text{e}^{i\gamma_O\phi}$.

In any case, consider all triangles in ${\cal C}_O^{\cal R}$ which are incident 
to $v_O$. The {\em aperture of $O$} is the sum 
of all the face angles at $O$ of the projection of these triangles. Denote by 
$\gamma_O$ the 
value $2\pi$ divided by the aperture. Note that for branch points we have 
$\gamma_O\in\{1/n: n=2,3,4,\dots\}$, so $\gamma_O\leq 1/2$.

We demand that the triangles in the neighborhood ${\cal 
C}_O$ of $O$ have an {\em adapted size}: as an additional condition, we demand 
that 
\begin{itemize}
\item the images under the chart $g_O$ of any two 
incident vertices in ${\cal C}_O$ have maximum distance $h$. 
\end{itemize}

In particular for $O\not=\infty$, consider any triangle 
$\Delta$ in $Pr_{\cal R}({\cal C}_O^{\cal R})$ whose vertex $z$ nearest to $O$ 
satisfies $|Oz|\geq 
h_O=h^{1/\gamma_O}$, where $|Oz|$ denotes the distance of $z$ to $O$ in $\C$. 
Then we deduce from our assumption that the maximal edge length in $\Delta$ 
is smaller than $h\cdot |Oz|^{1-\gamma_O}$.

In Section~\ref{SecPoly} we explain how our ideas can be used for polyhedral 
surfaces with more general 
conical singularities with $0<\gamma_O\leq 1/2$.

We will always assume that the maximal edge 
length $h$ is strictly smaller than $\max\{ \varrho/4, r_O/4, 1\}$.

A triangulation $T$ which satisfies these additional properties for all its 
branch points will be called {\em adapted triangulation}.

\begin{theorem}[Energy convergence]\label{lemConvEnergy}
 For each $\delta>0$ and each smooth multi-valued harmonic function 
$u:\widetilde{{\cal R}}\to \R$ there are two constants 
$\Const_{u,\delta,{\cal R},\varrho}, \const_{u,\delta,{\cal R},\varrho}>0$ 
such that for any 
adapted triangulation $T$ of ${\cal R}$ with maximal edge length 
$h<\const_{u,\delta,{\cal R},\varrho}$ and minimal face angle $>\delta$ we have
\[|E_T(u|_{\widetilde{V}})-E(u)|\leq \Const_{u,\delta,{\cal R},\varrho}\cdot 
h.\]
\end{theorem}

The assumption on the minimal face angle in the theorem cannot be dropped, 
see~\cite[Example~4.14]{BoSk16}.

Based on energy estimates from this theorem, we deduce convergence of discrete 
period matrices. To this end, recall that ${\cal R}$ is a Riemann surface which 
is a branched covering of $\hC$. Therefore, a {\em basis of 
holomorphic integrals} $\phi^l_{\cal R}:\widetilde{{\cal R}}\to 
\C$ and the {\em 
period matrix} $\Pi_{\cal R}$ of ${\cal R}$ are defined analogously to the 
discrete case above.

\begin{theorem}[Convergence of period matrices]\label{theoPeriodConv}
 For each $\delta>0$ there are two constants 
$\Const_{\delta,{\cal R},\varrho}$, $\const_{\delta,{\cal R},\varrho}>0$ such 
that 
for any adapted triangulation $T$ of ${\cal R}$ with maximal edge length 
$h<\const_{u,\delta,{\cal R},\varrho}$ and minimal face angle $>\delta$ we have
\begin{equation}\label{eqEstPi}
\|\Pi_T-\Pi_{\cal R}\|\leq \Const_{\delta,{\cal R},\varrho}\cdot h.
\end{equation}
\end{theorem}

Both theorems are proved in Section~\ref{SecConvProof}.

\subsection{Convergence of discrete holomorphic integrals}\label{SecConvAbel}

For the next theorem, we need some additional notions similarly as 
in~\cite{BoSk16}. The discrete holomorphic
integral $\phi^l_T=(\text{Re} \phi^l_T :\widetilde{V}\to\R, \;
\text{Im} \phi^l_T:\widetilde{F}\to\R)$ is {\em normalized} at a vertex 
$z\in\widetilde{T}$ and a face $w\in\widetilde{F}$, if $\text{Re} 
\phi^l_T(z)=0=\text{Im} \phi^l_T(w)$. Similarly, we call a holomorphic 
integral $\phi^l_{\cal R}:\widetilde{{\cal R}}\to \C$ normalized at a point 
$z\in\widetilde{{\cal R}}$ if $\phi^l_{\cal R}(z)=0$.

Recall that a triangulation $T$ is {\em Delaunay}, if for every edge $e\in E$ 
we have $\alpha_e+\beta_e\leq \pi$.

Let $\{ T_n\}$ be a sequence of adapted triangulations of the surface ${\cal 
R}$ with maximal edge length $h<\max\{\varrho/4, r_0/4, 1 \}$. Such a sequence 
of adapted triangulations is called {\em non-degenerate uniform}, if there is a 
constant $\Const$, not depending on $n$, such that for each member of the 
sequence:
\begin{enumerate}
 \item[(A)] the angles of each triangle are greater than $1/\Const$.
\item[(D)] for each edge the sum of opposite angles in the two triangles 
containing  the edge is less than $\pi-1/\Const$. (In particular, the 
triangulation is Delaunay within $B_{\varrho}(0)$.)
\item[(U)] the number of vertices in an arbitrary intrinsic disk about $z$ of 
radius equal to the maximal edge length is smaller than $\Const$ if $z$ is not 
contained in any of the neighborhoods ${\cal C}_O$ of a singularity $O$. Within 
such a neighborhood ${\cal C}_O$, first map the vertices to a disc about the 
origin by the map $\zeta\mapsto (\zeta-O)^{\gamma_O}$ if $O\not=\infty$ (or 
$\zeta\mapsto 1/\zeta^{\gamma_O}$ if $O=\infty$). Then we require that 
after this mapping 
in each disk of radius equal to the maximal edge length the number of image 
points of vertices is smaller than $\Const$.
\end{enumerate}

A sequence of functions $f_n=(\text{Re} f_n :\widetilde{V}_n\to\R, \;
\text{Im} f_n:\widetilde{F}_n\to\R)$ {\em converges} to a function 
$f:\widetilde{{\cal R}}\to\C$ {\em uniformly on compact subsets} if for every 
compact set $K\subset\widetilde{{\cal R}}$ we have
\begin{equation*}
\max\limits_{z\in K\cap \widetilde{V}_n } |\text{Re} f_n(z)-\text{Re} f(z)| 
\to 0\quad \text{and}\
 \max\limits_{\substack{\Delta[x,y,z]\in F,\\ \Delta[x,y,z]\cap 
K\not=\emptyset}}
 |\text{Im} f_n(\Delta[x,y,z])-\text{Im} f(z)| \to 0 \quad\text{as }  
n\to\infty.
\end{equation*}

\begin{theorem}[Convergence of holomorphic integrals]\label{theoConvAbelInt}
 Let $\{ T_n\}$ be a sequence of non-degenerate uniform adapted triangulations 
of ${\cal R}$ with maximal edge length $h_n\to 0$ as $n\to\infty$. Let $z_n\in 
\widetilde{V}_n$ be a sequence of vertices converging to a point 
$z_0\in\widetilde{\cal R}$. 
Let $\Delta_n\in \widetilde{F}_n$ be a sequence of faces with its vertices 
converging to $z_0$. Then for each $1\leq l\leq g$ the discrete holomorphic
integrals $\phi^l_T=(\text{Re} \phi^l_T :\widetilde{V}\to\R,\; 
\text{Im} \phi^l_T:\widetilde{F}\to\R)$ normalized at $z_n$ and $\Delta_n$ 
converge uniformly on each compact set to the holomorphic integral 
$\phi^l_{\cal R}:\widetilde{{\cal R}}\to\C$ normalized at $z_0$.
\end{theorem}

This theorem is proved in Section~\ref{SecConvAbelProof}.

\section{Proof of convergence of energy and period matrices}\label{SecConvProof}

In this section, we prove convergence of the discrete energy to the 
corresponding 
Dirichlet energy and convergence of discrete period matrices to their 
continuous counterparts. The main ideas of the proof 
follow~\cite[Section~4]{BoSk16}, but we improve the estimates near branch 
points (which can be considered as special conical singularities) by using the 
additional properties of the adapted triangulations. 

We denote by $\Const_{a,b,c}$ a 
positive constant which only depends on the parameters $a,b,c$. 
The symbol $\Const$ may denote distinct constants at different places of the 
text, for example in $2\cdot \Const\leq \Const$. 
Furthermore, we set $\|D^ku(z)\|:=\max_{0\leq j\leq k}\left| \frac{\partial^k 
u}{\partial^jx\partial^{k-j}y}(z)\right|$.

In the following, all triangle which are considered are in $\C$ after 
application of $Pr_{\cal R}$.

\subsection{Energy estimates in a triangle}\label{SecEstTriangle}
First we consider only one triangle $\Delta$ of the triangulation 
$T$. Let $u:\Delta\to\R$ be a smooth function which smoothly extends to a 
neighborhood of $\Delta$. Let $I_Tu:\Delta\to\R$ be the corresponding 
interpolation function defined in Section~\ref{secDiscHarm}. Then we set
$E_{\Delta}(u)=\int_\Delta |\nabla u|^2 dxdy$ and 
$E_{T_\Delta}(u)=\int_\Delta |\nabla I_Tu|^2 dxdy$.
Denote by $\delta$ the minimal angle of the triangle $\Delta$.

\begin{lemma}[Energy approximation on a triangle]\label{lemTriangEst}
\begin{enumerate}[(i)]
 \item If the triangle $\Delta$ is contained in $B_{\varrho}$, denote by 
$l_{max}$ the maximum edge length of 
$\Delta$. Then
\begin{multline*}
|E_{T_\Delta}(u)-E_{\Delta}(u)| \\
\leq \Const_\delta \cdot \left( \max\limits_{w\in\Delta}\|D^1u(w)\| 
+l_{max}\cdot \max\limits_{w\in\Delta}\|D^2u(w)\| \right) \cdot 
l_{max}\cdot \max\limits_{w\in\Delta}\|D^2u(w)\| \cdot \area(\Delta).
\end{multline*}
\item If the triangle $\Delta$ is contained in the complement 
$\hC\setminus B_{\varrho}$ of $B_{\varrho}$, denote $\tilde{u}=u\circ 1/z$ 
and let $\widetilde{\Delta}$ be the image of $\Delta$ under the map $1/z$. 
Further, let $\tilde{l}_{max}$ the maximum edge length of the image 
triangle $\widetilde{\Delta}$. Then 
\begin{multline*}
|E_{T_\Delta}(u)-E_{\Delta}(u)| \\
 \leq \Const_\delta \cdot \left( 
\max\limits_{w\in\widetilde{\Delta}} \|D^1\tilde{u}(w)\| 
+\tilde{l}_{max}\cdot \max\limits_{w\in\widetilde{\Delta}} 
\|D^2\tilde{u}(w)\| 
\right) \cdot \widetilde{l}_{max}\cdot \max\limits_{w\in\widetilde{\Delta}} 
\|D^2\tilde{u}(w)\| \cdot \area(\widetilde{\Delta}).
\end{multline*}
\item If the triangle $\Delta$ is a boundary triangle in $F_{\varrho}$ then
\[|E_{T_\Delta}(u)-E_{\Delta}(u)| \leq \Const_{\delta,\varrho} \cdot 
\max\limits_{w\in\Delta}\|D^1u(w)\|^2\cdot \area(\Delta)\]
\end{enumerate}
\end{lemma}
\begin{proof}
\begin{enumerate}[(i)]
 \item For triangles contained in $B_{\varrho}$ this is Lemma~4.4 
in~\cite{BoSk16}.
\item For triangles contained in $\hC\setminus B_{\varrho}$, note that the 
discrete energy is actually defined using the image of $\Delta$ under the map 
$1/z$ and the corresponding function $\tilde{u}=u\circ 1/z$. By 
conformality, the smooth
Dirichlet energy can also be considered on $\widetilde{\Delta}$. Therefore, the 
same arguments as in~(i)
apply.
\item For boundary triangles we estimate the discrete and the smooth energy 
separately. For the smooth energy, we have  $E_{\Delta}(u)\leq \Const \cdot 
\max\limits_{w\in\Delta}\|D^1u(w)\|^2\cdot \area(\Delta)$.
For the discrete energy, an estimate of $E_{T_\Delta}(u)$ is contained in  
Appendix~\ref{app2}, see in particular~\eqref{eqestEDelta}.
\end{enumerate}
\end{proof}

\subsection{Energy estimates near a a branch point}\label{SecEstEnergy}

Let $v_O\in\cal R$ be a branch point of $\cal R$ with $\gamma_O\leq 1/2$. 
In this subsection we only consider those triangles of the adapted 
triangulation $T$ which are completely contained in the neighborhood ${\cal 
C}_O^{\cal R}$. Denote by $T_O$ the connected component of these triangles 
which 
contains $v_O$. For the estimate of the difference of energies for these 
triangles we consider in particular $E_{T_O}(u)=\sum_{\Delta\in F_O} 
E_{T_\Delta}(u)$ and $E_{S_O}(u)=\sum_{\Delta\in F_O} 
E_{\Delta}(u)$, where $F_O$ denotes the set of triangles in $T_O$ and $S_O$ is 
the neighborhood of $v_O$ covered by these triangles.

If $O=\infty$, we compose as above the map $Pr_{\cal R}$ with 
$z\mapsto 1/z$ and consider the corresponding function $u\circ 1/z$ instead of 
$u$. Then 
the following reasoning also applies to the image branch point at the 
origin and its $1/(2\varrho)$-neighborhood.

As the partial derivatives of $u$ (considered in a chart) are not necessarily 
bounded near the vertex $O$, we consider triangles in a 'very small' 
neighborhood of $O$ separately. Let ${S}_{O,h_O}$ be the union of faces of 
$T_O$ whose images under $Pr_{\cal R}$ intersect the disc of radius $h_O:= 
h^{1/\gamma_O}$ about $O$ and let 
$T_{O,h_O}$ be the restriction of $T_O$ to ${S}_{O,h_O}$. Denote by $F_{O,h_O}$ 
the set of faces of $T_{O,h_O}$.
%
Note that we use polar coordinates $(\rho,\phi)$ as introduced 
in Subsection~\ref{subsecConv1} as a chart for ${\cal C}_O^{\cal R}$.


\begin{lemma}[Derivative Estimation Lemma, 
{\cite[Lemma~4.5]{BoSk16}}]\label{lemDerivEst}
For each $w=(\rho,\phi)\in {\cal C}_O^{\cal R}$ such that $w\not=O$ we have 
\[\|D^1u(z)|_{w}\|\leq \const_{u,r_O,\gamma_O} \cdot \rho^{\gamma_O-1} \qquad 
\text{and}\qquad  \|D^2u(z)|_{w}\|\leq \const_{u,r_O,\gamma_O} \cdot 
\rho^{\gamma_O-2}.\]
\end{lemma}


\begin{lemma}[{\cite[Lemma~4.12]{BoSk16}}]\label{lemestaltETOh}
 For every $\Delta\in T_{O,h_O}$ we have $E_{T_\Delta}(u)\leq 
\Const_{\delta,\gamma_O,r_O,u}\cdot \int_\Delta \rho^{2\gamma_0-1}d\rho d\phi$.
\end{lemma}

\begin{lemma}\label{lemestEtriang}
For every $\Delta\in T_O\setminus T_{O,h_O}$ we have 
$|E_{T_\Delta}(u)-E_{\Delta}(u)| \leq \Const_{\delta,\gamma_O,r_O,u} \cdot h 
\cdot \int_\Delta \rho^{\gamma_O-1} d\rho d\phi$.
\end{lemma}
\begin{proof}
 Let $z\in\Delta$ be the vertex closest to $O$. As $\Delta\in F_O\setminus 
F_{O,h_O}$ we have for each point $(\rho,\phi)$ in $\Delta$ that 
$h^{1/\gamma_O}\leq |Oz|\leq \rho\leq |Oz|+h|Oz|^{1-\gamma_O} \leq 2|Oz|$. 
Furthermore, on $F_O\setminus F_{O,h_O}$ we can use 
Lemma~\ref{lemDerivEst} and Lemma~\ref{lemTriangEst} to obtain
\begin{align*}
|E_{T_\Delta}(u)-E_{\Delta}(u)| &\leq \Const_{\delta,\gamma_O,r_O,u} \left( 
|Oz|^{\gamma_O-1}+ h|Oz|^{1-\gamma_O}\cdot |Oz|^{\gamma_O-2}\right) \cdot 
h|Oz|^{1-\gamma_O} \cdot \text{Area}(\Delta)\\
&\leq \Const_{\delta,\gamma_O,r_O,u}\cdot h \cdot \int_\Delta 
\rho^{\gamma_O-1}d\rho d\phi\ .
\end{align*}
\end{proof}

\begin{lemma}\label{lemestEO}
 We have $|E_{T_O\setminus T_{O,h_O}}(u)-E_{S_O \setminus {S}_{O,h_O}}(u)|\leq 
\Const_{\delta,\gamma_O,r_O,u}\cdot h$.
\end{lemma}
\begin{proof}
We use Lemma~\ref{lemestEtriang}, sum the 
inequalities and estimate the integral.
\begin{align*}
 |E_{T_O\setminus T_{O,h_O}}(u)-E_{S_O \setminus {S}_{O,h_O}}(u)| &\leq 
\sum\limits_{\Delta\in F_O\setminus F_{O,h_O}} |E_{T_\Delta}(u)-E_{\Delta}(u)| 
\\
& \leq \Const_{\delta,\gamma_O,r_O,u}\cdot h \cdot \int_{S_O\setminus  
{S}_{O,h_O}} \rho^{\gamma_O-1}d\rho 
 \leq \Const_{\delta,\gamma_O,r_O,u}\cdot h
\end{align*}
\end{proof}

Now we estimate the energies on $S_{O,h_O}$ and $F_{O,h_O}$ separately.

\begin{lemma}\label{lemestESOh}
We have $E_{S_{O,h_O}}(u)\leq \Const_{\gamma_O,u}\cdot h$.
\end{lemma}
\begin{proof}
Using Lemma~\ref{lemDerivEst} and our definition of $S_{O,h_O}$ we obtain
\[ E_{S_{O,h_O}}(u)=\sum_{\Delta\in F_{O,h_O}} \int_\Delta |\nabla u|^2 dxdy
 \leq \Const_{\gamma_O,u} \int_{\phi=0}^{2\pi/\gamma_O} 
\int_{\rho=0}^{2^{1/\gamma_O}h^{1/\gamma_O}} \rho^{2\gamma_O-1} d\rho d\phi 
\leq \Const_{\gamma_O,u}\cdot h^2. \]
\end{proof}

\begin{lemma}\label{lemestETOh}
We have $E_{T_{O,h_O}}(u)\leq \Const_{\delta,\gamma_O,u}\cdot h$.
\end{lemma}
\begin{proof}
 We deduce from Lemma~\ref{lemestaltETOh} similarly as in the previous lemma 
that
\begin{align*}
E_{F_{O,h_O}}(u)&=\sum_{\Delta\in F_{O,h_O}} \int_\Delta |\nabla I_Tu|^2 dxdy \\
&\leq \Const_{\delta,\gamma_O,u} \int_{\phi=0}^{2\pi/\gamma_O} 
\int_{\rho=0}^{2^{1/\gamma_O}h^{1/\gamma_O}} \rho^{2\gamma_O-1} d\rho d\phi 
\leq \Const_{\delta,\gamma_O,u}\cdot h^2. 
\end{align*}
\end{proof}

\subsection{Convergence of energies}

 Let $G_T=F\setminus \{F_{\varrho}\cup \bigcup_{O\text{ branch point}} F_O\}$ 
be the set of triangles which are neither contained in the neighborhood of any 
branch point nor are boundary triangles. Denote by $G$ the subset of $\cal R$ 
covered by the triangles in $G_T$. 

\begin{lemma}\label{lemestG}
We have $|E_{G_T}(u)-E_{G}(u)| \leq \Const_{\delta,u,{\cal R},\varrho}\cdot h$.
\end{lemma}
\begin{proof}
 We split the set $G_T$ into two parts 
$G_T=G_T^{(b)}\cup G_T^{(i)}$ such that $Pr_{\cal R}$ maps all triangles of 
$G_T^{(b)}$ into the intrinsic disc $B_\varrho(0)$ of radius $\varrho$ about 
the origin and all triangles of $G_T^{(i)}$ into the complement.
Denote by $G^{(b)}$, $G^{(i)}\subset G$ the subsets covered by the 
triangles in $G_T^{(b)}$, $G_T^{(i)}$ respectively. We consider the energies on 
both parts separately.

Our assumption 
on the maximal edge length, the definition of the discrete energy, the 
compactness of $\cal R$, and the estimates in Lemma~\ref{lemTriangEst} imply 
that
\begin{align*}
|E_{G_T^{(b)}}(u)-E_{G^{(b)}}(u)| &\leq 
\sum\limits_{\Delta\in G_T^{(b)}} |E_{T_\Delta}(u)-E_{\Delta}(u)| 
\leq \sum\limits_{\Delta\in G_T^{(b)}} \Const_{\delta,u,\varrho}\cdot \left(1+ 
h\right)\cdot h \cdot \text{Area}(\Delta)\\
&\leq \sum\limits_{\Delta\in G_T^{(b)}} \Const_{\delta,u,\varrho}\cdot h \cdot 
\int_\Delta \rho d\rho d\phi\ 
 \leq \Const_{\delta,u,\varrho}\cdot h \cdot \int_{G^{(b)}}
\rho d\rho \\
& \leq \Const_{\delta,u,{\cal R},\varrho}\cdot h.
\end{align*}

For all triangles $\Delta\in G_T^{(i)}$ we consider the image 
$\widetilde{\Delta}$ under the map $z\to 1/z$.
Using the corresponding map $\widetilde{u}=u\circ 1/z$ we obtain analogously
\begin{align*}
|E_{G_T^{(i)}}(u)-E_{G^{(i)}}(u)| &\leq 
\sum\limits_{\Delta\in G_T^{(i)}} |E_{T_{\widetilde{\Delta}}}(\tilde{u}) 
-E_{\widetilde{\Delta}}(\tilde{u})| \\
&\leq \sum\limits_{\Delta\in G_T^{(i)}} \Const_{\delta,\tilde{u},\varrho} \cdot
\left(1+ h\right)\cdot h \cdot \text{Area}(\widetilde{\Delta})
\cdot 
 \leq \Const_{\delta,u,{\cal R},\varrho}\cdot h.
\end{align*}
\end{proof}

\begin{lemma}\label{lemestSrho}
 We have $E_{S_{\varrho}}(u)\leq \Const_{u,\varrho}\cdot h$, where 
$S_{\varrho}$ 
denotes the subset of $\cal R$ which is covered by triangles of $F_{\varrho}$.
\end{lemma}
\begin{proof}
 This estimate is due to the fact that the derivative of $u$ is bounded away 
from the branch points. Furthermore, the area of the ring $\{\varrho-h\leq 
|z|\leq \varrho+h\}$ is bounded by $4\pi\varrho h$ and the degree of the 
branched 
covering is fixed. Therefore,
\[E_{S_{\varrho}}(u)=\int\limits_{S_{\varrho}} |\nabla u|^2 dx dy \leq 
\Const_{u,\varrho,{\cal R}}\cdot h.\]
\end{proof}

\begin{lemma}\label{lemEstF0}
 We have $E_{F_{\varrho}}(u):=\sum_{\Delta\in F_{\varrho}} E_{T_\Delta}(u)\leq 
\Const_{u,\delta,\varrho,{\cal R}}\cdot h$.
\end{lemma}
The proof of this lemma is given in Appendix~\ref{app2}.

\begin{proof}[Proof of Theorem~\ref{lemConvEnergy}]
Summing up the estimates obtained in Lemmas~\ref{lemestEO}--\ref{lemEstF0} we 
get the desired result:
\begin{align*}
 |E_T(u|_{\widetilde{V}})-E(u)| &\leq |E_{G_T}(u)-E_{G}(u)| 
+E_{F_{\varrho}}(u) \\
&\quad +\sum\limits_{O \text{ branch point of }{\cal R}} (|E_{T_O\setminus 
T_{O,h_O}}(u)-E_{S_O \setminus {S}_{O,h_O}}(u)| 
+E_{S_{O,h_O}}(u)+E_{T_{O,h_O}}(u)) \\
& \leq \Const_{\delta,u,{\cal R},\varrho}\cdot h.
\end{align*}
\end{proof}

\subsection{Convergence of discrete period matrices}

For our convergence proof we start with some further useful theorems and 
definitions.

\begin{lemma}[Variational principle 
{\cite[Lemma~3.6]{BoSk16}}]\label{lemdiscVari}
 A multi-valued discrete harmonic function has minimal energy among all 
multi-valued functions with the same periods.
\end{lemma}

\begin{theorem}[{\cite[Theorem~3.9]{BoSk16}}]\label{theoAbelExist}
For each $P=(A_1,\dots, A_g,B_1,\dots, B_g)\in\R^{2g}$ there exists a 
unique (up to a constant) discrete holomorphic integral
$\phi_{T,P}=(\text{Re} \phi_{T,P} :\widetilde{V}\to\R, \;
\text{Im} \phi_{T,P}:\widetilde{F}\to\R)$ whose periods have {\em real parts} 
$A_1,\dots, A_g,B_1,\dots, B_g$, respectively.
\end{theorem}

Denote $u_{T,P}=\text{Re} \phi_{T,P}$, where $\phi_{T,P}$ is the discrete 
holomorphic integral defined in Theorem~\ref{theoAbelExist} for each 
vector $P\in\R^{2g}$. Analogously, let $\phi_{{\cal R},P}:\widetilde{\cal 
R}\to\C$ 
be a holomorphic integral whose periods have  {\em real parts} 
$A_1,\dots, A_g,B_1,\dots, B_g$, respectively. Denote $u_{{\cal R},P}=\text{Re} 
\phi_{{\cal R},P}$.

\begin{lemma}\label{lemEnergyFormConv}
 For every $\delta>0$ and every vector $P\in\R^{2g}$ there are constants 
$\Const_{P,\delta,{\cal R},\varrho}$, $\const_{P,\delta,{\cal R},\varrho}>0$ 
such that 
for any adapted triangulation $T$ of $\cal R$ with maximal edge length 
$h<\const_{P,\delta,{\cal R},\varrho}$ and minimal face angle $\delta>0$ we 
have 
\[ |E_T(u_{T,P})-E(u_{{\cal R},P})| \leq \Const_{P,\delta,{\cal R}, 
\varrho}\cdot 
h.\]
\end{lemma}
\begin{proof}
From the interpolation lemma~\ref{lemInterpolation} we know that 
$E_T(u_{T,P})=E(I_Tu_{T,P})$ 
and the interpolation function $I_Tu_{T,P}$ is continuous and piecewise smooth. 
Using Lemma~\ref{lemdiscVari} and its smooth counterpart (Dirichlet's 
principle) we deduce from Theorem~\ref{lemConvEnergy} that
\[0\leq E(I_Tu_{T,P})-E(u_{{\cal R},P})=E_T(u_{T,P})-E(u_{{\cal R},P}) \leq 
 E_T(u_{{\cal R},P}|_{\widetilde{V}}) -E(u_{{\cal R},P}) \leq 
\Const_{P,\delta,{\cal 
R},\varrho}\cdot h.\]
\end{proof}

For each $l=1,\dots,g$ denote by $\phi^l_{T^*}=(\text{Re} \phi^l_{T^*} 
:\widetilde{V}\to\R, \; \text{Im} \phi^l_{T^*}:\widetilde{F}\to\R)$ the unique 
(up to constant) discrete 
holomorphic integral with $A$-periods given by 
$A_k=i\delta_{kl}$, 
where $k=1,\dots,g$. The $g\times g$-matrix $\Pi_{T^*}$ whose $l$-th column is
formed by the $B$-periods of $\phi^l_{T^*}$ divided by $i$, where 
$l=1,\dots,g$, is called the 
{\em dual period matrix} of the triangulation $T$.

The following theorem connects the period matrices to the energies.
\begin{lemma}[{\cite[Lemmas~3.14 \&~3.15]{BoSk16}}]\label{lemEquad}
 \begin{enumerate}[(i)]
  \item The energy $E_T(u_{T,P})$ is a quadratic form in the vector 
$P\in\R^{2g}$ with the block matrix
\[ E_T:= \begin{pmatrix} \mathrm{Re} \Pi_{T^*} (\mathrm{Im}\Pi_{T^*})^{-1} 
\mathrm{Re} 
\Pi_{T}+\mathrm{Im}\Pi_{T} & -(\mathrm{Im}\Pi_{T^*})^{-1} \mathrm{Re} \Pi_{T} \\
-\mathrm{Re} \Pi_{T^*} (\mathrm{Im}\Pi_{T^*})^{-1} & 
(\mathrm{Im}\Pi_{T^*})^{-1} \end{pmatrix} .\]
\item The energy $E(u_{{\cal R},P})$ is a quadratic form in the vector 
$P\in\R^{2g}$ with the block matrix
\[ E_{\cal R}:= \begin{pmatrix} \mathrm{Re} \Pi_{\cal R} 
(\mathrm{Im}\Pi_{\cal R})^{-1} \mathrm{Re} 
\Pi_{\cal R}+\mathrm{Im}\Pi_{\cal R} & -(\mathrm{Im}\Pi_{\cal R})^{-1} 
\mathrm{Re} 
\Pi_{\cal R} \\
-\mathrm{Re} \Pi_{\cal R} (\mathrm{Im}\Pi_{\cal R})^{-1} & 
(\mathrm{Im}\Pi_{\cal R})^{-1} \end{pmatrix} .\]
 \end{enumerate}
\end{lemma}

Combining Lemmas~\ref{lemEnergyFormConv} and~\ref{lemEquad}, we obtain:

\begin{corollary}\label{corConvE}
 Let $\{T_n\}$ be a nondegenerate uniform sequence of adapted triangulations of 
$\cal R$ with maximal edge length tending to zero as $n\to\infty$. Let 
$P_n\in\R^{2g}$ be a sequence of $2g$-dimensional real vectors converging to a 
vector $P\in\R^{2g}$. Then $E_{T_n}(u_{T_n,P_n})\to E(u_{{\cal R},P})$ as 
$n\to\infty$.
\end{corollary}

\begin{proof}[Proof of Theorem~\ref{theoPeriodConv}.]
 Both $E_{T_n}(u_{T_n,P})$ and $E(u_{{\cal R},P})$ are quadratic forms in 
$P\in\R^{2g}$ by Lemma~\ref{lemEquad} with block matrices $E_T$ and $E_{\cal 
R}$, respectively. Thus by Lemma~\ref{lemEnergyFormConv} for every $\delta>0$ 
there are constants 
$\Const_{\delta,{\cal R},\varrho},\const_{\delta,{\cal R},\varrho}>0$ such that 
for any adapted triangulation $T$ of $\cal R$ with maximal edge length 
$h<\const_{\delta,{\cal R},\varrho}$ and minimal face angle $\delta>0$ we 
have $\|E_T-E_{\cal R}\| \leq \Const_{\delta,{\cal R}, \varrho}\cdot h$. From 
this inequality we deduce estimates on $\|\text{Re} \Pi_{T} -\text{Re} 
\Pi_{\cal R}\|$ and $\|\text{Im} \Pi_{T} -\text{Im} \Pi_{\cal R}\|$ of the 
same type, but with different constants which are derived in the following. 
These estimates complete the proof.
\begin{itemize}
 \item As $\|(\text{Im}\Pi_{T^*})^{-1}- (\text{Im}\Pi_{\cal R})^{-1}\|\leq 
\Const_{\delta,{\cal R}, \varrho}\cdot h$ for $h<\const_{\delta,{\cal 
R},\varrho}$ 
there exist new constants $\Const_{\delta,{\cal R}, \varrho}'>0$ and 
$\const_{\delta,{\cal R},\varrho}>\const_{\delta,{\cal R},\varrho}'>0$ such 
that 
$\|\text{Im}\Pi_{T^*}\|\leq \Const_{\delta,{\cal R}, \varrho}'$ for 
$h<\const_{\delta,{\cal R},\varrho}'$.
\item Thus for $h<\const_{\delta,{\cal R},\varrho}'$ we deduce
\begin{align*}
 \Const_{\delta,{\cal R}, \varrho}\cdot h \geq & \| (\text{Im}\Pi_{T^*})^{-1} 
\text{Re} \Pi_{T} -(\text{Im}\Pi_{\cal R})^{-1} \text{Re} \Pi_{\cal R}\| \\
&= \| (\text{Im}\Pi_{T^*})^{-1} (\text{Re} \Pi_{T}- \text{Re} \Pi_{\cal R}) 
-((\text{Im}\Pi_{\cal R})^{-1}- (\text{Im}\Pi_{T^*})^{-1} ) \text{Re} \Pi_{\cal 
R}\| \\
&\geq \|(\text{Im}\Pi_{T^*})^{-1}\|\cdot \|\text{Re} \Pi_{T} -\text{Re} 
\Pi_{\cal R}\|- \|(\text{Im}\Pi_{T^*})^{-1}- (\text{Im}\Pi_{\cal 
R})^{-1}\|\cdot \|\text{Re} \Pi_{\cal R}\| \\
&\geq (\Const_{\delta,{\cal R}, \varrho}')^{-1}\cdot \|\text{Re} \Pi_{T} 
-\text{Re} \Pi_{\cal R}\|- \Const_{\delta,{\cal R}, \varrho}\cdot h \cdot 
\|\text{Re} \Pi_{\cal R}\| .
\end{align*}
Therefore, $\|\text{Re} \Pi_{T} -\text{Re} \Pi_{\cal R}\| \leq 
\Const_{\delta,{\cal R}, \varrho}''\cdot h$, where $\Const_{\delta,{\cal R}, 
\varrho}''=\Const_{\delta,{\cal R}, \varrho}'\cdot \Const_{\delta,{\cal R}, 
\varrho} 
\cdot (1+\|\text{Re} \Pi_{\cal R}\|)$.

Analogously, we see that $\|\text{Re} \Pi_{T^*} -\text{Re} \Pi_{\cal R}\| \leq 
\Const_{\delta,{\cal R}, \varrho}''\cdot h$.
\item By similar estimates as for the previous item, we obtain
\[\|\text{Re} \Pi_{T^*} (\text{Im}\Pi_{T^*})^{-1} \text{Re} \Pi_{T}- \text{Re} 
\Pi_{\cal R} (\text{Im}\Pi_{\cal R})^{-1} \text{Re} \Pi_{\cal R}\| \leq 
\Const_{\delta,{\cal R}, \varrho}'''\cdot h,\]
where $\Const_{\delta,{\cal R}, \varrho}'''= \Const_{\delta,{\cal R}, 
\varrho}'' 
\cdot \|(\text{Im}\Pi_{\cal R})^{-1}\|(1+2\|\text{Re} \Pi_{\cal R}\|) + 
\Const_{\delta,{\cal R}, \varrho} \cdot (\Const_{\delta,{\cal R}, \varrho}'' 
+\|\text{Re} \Pi_{\cal R}\|)^2$.
\item Finally, we deduce from
\[\|\text{Re} \Pi_{T^*} (\text{Im}\Pi_{T^*})^{-1} \text{Re} \Pi_{T} 
+\text{Im}\Pi_{T}- \text{Re} \Pi_{\cal R} 
(\text{Im}\Pi_{\cal R})^{-1} \text{Re} \Pi_{\cal R} -\text{Im}\Pi_{\cal R}\| 
\leq \Const_{\delta,{\cal R}, \varrho}\cdot h\]
together with the previous estimate that $\|\text{Im} \Pi_{T} -\text{Im} 
\Pi_{\cal R}\| 
\leq (\Const_{\delta,{\cal R}, \varrho}''' +\Const_{\delta,{\cal R}, \varrho}) 
\cdot h$. 
\end{itemize}
\end{proof}

\section{Proof of convergence of discrete holomorphic 
integrals}\label{SecConvAbelProof}

The strategy of the proof of Theorem~\ref{theoConvAbelInt} follows the 
corresponding ideas in~\cite[Section~5]{BoSk16}. Due to our different setup, we 
need some modifications.

\subsection{Equicontinuity}\label{SecEqui}

In this section we consider triangulations $T'$ of branched coverings with 
boundary. The main goal is to consider (sufficiently small) intrinsic discs 
about a branch points or about a regular point and derive an estimate for 
harmonic functions there. A function $u:V'\to\R$ is discrete harmonic on $T'$ 
if it satisfies~\eqref{eqdefharm} at every non-boundary vertex. 
Denote $E'_{T'}(u)=\sum\limits_{e=[x,y]\in E'\setminus\partial E'} 
c([x,y])(u(y)-u(x))^2$, where the sum is over non-boundary edges. Let the 
{\em eccentricity} $e$ denote the number $\Const$ such the triangulation $T$  
satisfies conditions (A), (D), (U) from Section~\ref{SecConvAbel}, where (A) 
and (D) only hold for every non-boundary edge.

Let $T$ be a non-degenerate uniform adapted triangulation of the branched 
covering of $\cal R$. We assume that $T'$ is a simply connected part of $T$. 
For simplicity, we directly consider the projection of all triangles into $\C$ 
by $Pr_{\cal R}$.

\begin{lemma}[Equicontinuity lemma]\label{lemEqui}
\begin{enumerate}[(i)]
 \item Let $T'$ be contained in an open disc $B_r(v)\subset\C$ where 
$2r$ is smaller than the minimum distance of $v$ to any branch point, but 
$r\geq 10\cdot h$. 
Denote by $h'$ twice the maximum circumradius of the triangles of $T'$. Let 
$u:V'\to\R$ be a discrete harmonic function. Let $z,w\in V'$ with Euclidean 
distance $|z-w|\geq h'$ and such that $3|z-w|<r<dist(zw,\partial T')$ for some 
$r>0$. Here $dist(zw,\partial T')$ denotes the distance of the straight line 
segment from $z$ to $w$ to the boundary of $T'$. Then there exists a constant 
$\Const_{e}>0$ such that 
\begin{equation}\label{eqequiu1}
|u(z)-u(w)|\leq \Const_{e} \cdot E'_{T'}(u)^{1/2} \cdot \left(\log 
\frac{r}{3|z-w|} \right)^{-1/2}.
\end{equation}
For $|z-w|<h'<r/3$ the same inequality holds with $|z-w|$ replaced by $h'$.
\item Let $T'$ be contained in an open intrinsic disc $B_{r_O}(O)\subset\C$ 
about some branch points $O$. Let $u:V'\to\R$ be a discrete harmonic function.

Consider the chart $g_O(z)=(z-O)^{\gamma_O}$, which maps the triangulation $T'$ 
contained in ${\cal C}_O$ to an embedded triangulation $T'_g$ in a neighborhood 
of the origin in $\C$.
Denote by $h'$ twice the maximum circumradius of the triangles of $T'_g$. Let 
$z,w\in V'$ with Euclidean 
distance $|g_O(z)-g_O(w)|= |(z-O)^{\gamma_O}- (w-O)^{\gamma_O}| \geq h'$ and 
such that $3|(z-O)^{\gamma_O}- (w-O)^{\gamma_O}|<r<dist(g_O(z)g_O(w),\partial 
T'_g)$ for some $r>0$. Here $dist(g_O(z)g_O(w),\partial T')$ denotes the 
distance of the straight line segment from $g_O(z)$ to $g_O(w)$ to the boundary 
of $T'_g$. Then there exists a constant $\Const_{e}>0$ such that 
\begin{equation}\label{eqequiu2}
|u(z)-u(w)|\leq \Const_{e} \cdot E'_{T'}(u)^{1/2} \cdot \left(\log 
\frac{r}{3|(z-O)^{\gamma_O}- (w-O)^{\gamma_O}|} \right)^{-1/2}.
\end{equation}
For $|(z-O)^{\gamma_O}- (w-O)^{\gamma_O}|<h'<r/3$ the same inequality holds 
with $|(z-O)^{\gamma_O}- (w-O)^{\gamma_O}|$ replaced by $h'$.
\item Let $T'$ be contained in the open intrinsic disc 
$\hC\setminus B_{\varrho}(\infty)\subset\hC$. Then for the image $T'_{1/z}$ of 
$T'$ under the map $1/z$ the estimates in~(i) and~(ii) hold depending on 
whether $\infty$ is a regular point or a branch point of $\cal R$.
\end{enumerate}
\end{lemma}
\begin{proof}
The claims are proved analogously to a similar estimate for quadrilateral 
lattices in the plane~\cite[Equicontinuity Lemma~2.4]{Sko13}, see 
also~\cite[\S~1 and Remarks~3.4 and~4.8]{Sko13}, using the approach 
of~\cite[Section~5.4]{Lu26}. For the sake of completeness, we 
present a proof in Appendix~\ref{app3}.

In case~(ii), we consider the harmonic function $u$ as defined on the 
image triangulation $T'_g$. The proof  
only uses the fact that $u$ satisfies a maximum principle which still holds in 
our case. 
For the third case, we just work with the triangulation $T'_{1/z}$ and assume 
without loss of generality that $u$ is defined there.
\end{proof}

\begin{lemma}\label{lemboundE'}
 Let $T$ be a triangulation of a ramified covering with boundary such that all 
angles are in $[\delta,\pi-\delta]$ for some $\pi/4>\delta>0$. Then there 
exist constants $\const_{\delta,\varrho}, {\Const}_{\delta,\varrho}>0$ 
such that for $0<h<\const_{\delta,\varrho}$ and every function
$u:V\to\R$ we have $E_{T'}'(u)\leq \Const_{\delta,\varrho}\cdot E_T(u)$. 
\end{lemma}
\begin{proof}
Let $\Delta\in F'$ be a triangle with vertices $x,y,z\in V$ such that $[x,z]$ 
is 
a boundary edge of $T'$. Denote the angle in $\Delta$ at the vertex 
$v\in\{x,y,z\}$ by $\alpha_v$.

 First consider the case that $\Delta\not\in F_{\varrho}$ is no boundary 
triangle. We want to show that
\begin{align}\label{eqestEbound}
E_{T_\Delta}(u) &= \textstyle \frac{1}{2}\cot\alpha_x (u(y)-u(z))^2 
+\textstyle \frac{1}{2}\cot\alpha_z (u(x)-u(y))^2 +\textstyle 
\frac{1}{2}\cot\alpha_y (u(z)-u(x))^2 
\\ 
&\geq \Const_{\delta}\cdot |\textstyle \frac{1}{2}\cot\alpha_y| (u(z)-u(x))^2.
\end{align}
holds for some constant $\Const_{\delta}>0$. Thus
we only need to consider the case $\alpha_y>\pi/2$. Take 
$\Const_{\delta}=1/(\cot^2\delta-1)$. As $\alpha_x,\alpha_z>\delta$ and 
$\alpha_x +\alpha_y +\alpha_z=\pi$, elementary calculations imply that
\[0\leq 1+ \Const_{\delta}\cdot (1-\cot\alpha_x\cdot \cot\alpha_z )= 
\cot\alpha_x\cdot \cot\alpha_z +\cot\alpha_y(1+\Const_{\delta})(\cot\alpha_x 
+\cot\alpha_z).\]
This implies~\eqref{eqestEbound}.

If $\Delta\in F_{\varrho}$, we know that 
\[C_{[x,z]}=\cot\alpha^E_y +h\cdot r_y,\qquad
 C_{[z,y]}=\cot\alpha^E_x +h\cdot r_x,\qquad
C_{[y,x]}=\cot\alpha^E_z +h\cdot r_y,
\]
where $\alpha_v^E$ denotes the angle at the vertex $v$ in the Euclidean 
triangle with vertices $x,y,z$ and $|r_v|\leq \Const_{\delta,\varrho}$, see 
Appendix~\ref{secApp1}. Therefore, there are constants
$\const_{\delta,\varrho}, \widetilde{\Const}_{\delta,\varrho}>0$ such that for 
all 
$0<h<\const_{\delta,\varrho}$ we have $E_{T_\Delta}(u) \geq 
\widetilde{\Const}_{\delta,\varrho} |C_{[x,z]}| (u(z)-u(x))^2$.

Take $\Const_{\delta,\varrho}= \max\{ \widetilde{\Const}_{\delta,\varrho}, 
\Const_{\delta}\}$, sum the above inequalities over all such faces and deduce 
$E'_{T'}(u)- E_T(u)\leq \Const_{\delta,\varrho}\cdot E_T(u)$. Now the claim 
follows.
\end{proof}

\subsection{Convergence of multi-valued discrete harmonic functions and 
discrete holomorphic integrals}

As a first step, we can deduce that the uniform limit of a sequence of discrete 
harmonic functions is harmonic. To this end, we say that a sequence of 
triangulated polygons $\{T_n\}$ approximates a domain $\Omega\subset\C$, if for 
$n\to\infty$ the following three quantities tend  to zero: the maximal distance 
from a point of the boundary $\partial T_n$ to the set $\partial \Omega$, the 
maximal distance from a point of $\partial \Omega$ to the set $\partial T_n$, 
and the maximal edge length of the triangulation $T_n$.

\begin{lemma}[{\cite[Lemma~5.2]{BoSk16}}]\label{lemHarmLimit}
 Let $\{T_n\}$ be a non-degenerate uniform sequence of Delaunay triangulations 
of polygons with boundary approximating a domain $\Omega$, such that no 
branch point in on $\partial \Omega$. Let 
$u_n:V_n\to\R$ be a sequence of discrete harmonic functions uniformly
converging to a continuous function $u:\Omega\to\R$. Then the function 
$u:\Omega\to\R$ is harmonic.
\end{lemma}

\begin{theorem}[Convergence of multi-valued discrete harmonic 
functions]\label{theoConvmultiharmonic}
 Let $\{T_n\}$ be a non-degenerate uniform sequence of adapted Delaunay 
triangulations of $\cal R$ with maximal edge length $h_n$ tending to zero as 
$n\to\infty$. Let $z_n\in \widetilde{V}_n$ be a sequence of vertices converging 
to 
a point $z_0\in\widetilde{\cal R}$. Let $P_n\in\R^{2g}$ be a sequence of 
vectors 
converging to a vector $P\in\R^{2g}$. Then the functions 
$u_{T_n,P_n}:\widetilde{V}_n\to\R$ satisfying $u_{T_n,P_n}(z_n)=0$ converge to 
$u_{{\cal R},P}:\widetilde{\cal R}\to\R$ with $ u_{{\cal R},P}(z_0)=0$ 
uniformly on every compact subset.
\end{theorem}


\begin{proof}
 We will start with some estimates on compact subsets of $\widetilde{\cal R}$ 
of a special form. Let $\pi=Pr_{\cal R}\circ p:\widetilde{\cal R}\to{\C}$ be 
the local projection map $Pr_{\cal R}$ composed with the universal covering $p$.
For $v\in\widetilde{\cal R}$ denote by 
$\widetilde{B}_r(v)\subset \widetilde{\cal R}$ 
the subset which projects for $\pi(v)\in\C$ to an open intrinsic disc 
$B_r(\pi(v))=\pi(\widetilde{B}_r(v))$ with radius $r$ about $\pi(v)$. If 
$\pi(v)=\infty$, 
we assume that $\pi(\widetilde{B}_r(v))= \C\setminus B_{1/r}(0)$. We restrict 
ourselves to the following cases: 
\begin{itemize}
\item $\pi(v)=O$ is a branch point and 
$r=r_O(v)>0$ its associated radius defined in Section~\ref{subsecConv1}, 
\item $\pi(v)\in B_{\varrho}\setminus \bigcup\limits_{O \text{ branch point}} 
B_{r_O}(O)$ and $r_O^{min}/8<r\leq r_O^{min}/2$, where 
$r_O^{min}/2:=\min\limits_{O \text{ branch point}}r_O/2$,
\item $\pi(v)=\infty$ and $r=1/\varrho$.
\end{itemize}
Note that the union of these sets $\widetilde{B}_r(v)$ covers $\widetilde{\cal 
R}$ and every compact 
set $K\subset \widetilde{\cal R}$ is contained in the union of finitely many of 
these sets. 

Let $\widetilde{B}_r(v)$ be one of these sets. Consider those 
triangles of the given triangulation $\widetilde{T_n}$ which are completely 
contained in $\widetilde{B}_r(v)$ and denote by $\widetilde{T}_n(v,r)$ the 
connected 
component of these triangles which contains $v$. Choose $n_1$ such that for all 
$n>n_1$ the maximal edge length $h_n<r_O^{min}/200$. Consider 
$u_{\widetilde{V}_n(v,r)}:= u_{T_n,P_n}|_{\widetilde{V}_n(v,r)}$. By 
Lemma~\ref{lemboundE'} and Corollary~\ref{corConvE} the sequence of energies 
$E'_{\widetilde{T}_n(v,r)}(u_{\widetilde{V}_n(v,r)})$ is bounded. Thus the 
Equicontinuity lemma~\ref{lemEqui} implies that the function 
$u_{\widetilde{V}_n(v,r)}|_{\widetilde{V}_n\cap 
\widetilde{B}_{\frac{3}{4}r}(v)}$ has 
uniformly bounded differences. That is, there exists a constant $\Const_{{\cal 
R}, P,\delta}$ such that for all $n>n_1$ and $z,w\in \widetilde{V}_n\cap 
\widetilde{B}_{\frac{3}{4}r}(v)$ we have $|u_{T_n,P_n}(z)- u_{T_n,P_n}(w)|\leq 
\Const_{{\cal R}, P,\delta}$. Lemma~\ref{lemEqui} also implies that the 
sequence 
is equicontinuous, that is, there exists a function $\delta(\varepsilon)$ for 
$\varepsilon>0$ such that for each  $n>n_1$ and $z,w\in \widetilde{V}_n\cap 
\widetilde{B}_{\frac{3}{4}r}(v)$ with $|z-w|<\delta(\varepsilon)$ we have 
$|u_{T_n,P_n}(z)- u_{T_n,P_n}(w)|\leq \varepsilon$.

Now take a sequence of compact sets $K_1\subset K_2\subset \dots \subset 
\widetilde{\cal R}$ such that $\widetilde{\cal R}=\bigcup_{j=1}^\infty K_j$. 
Assume that $K_1$ contains all point of the convergent sequence $\{z_n\}$. 
Since $K_1$ is compact, it is contained in the union of finitely many of the 
sets 
considered above. Therefore, the sequence $u_{T_n,P_n}|_{\widetilde{V}_n\cap 
K_1}$ is equicontinuous and has uniformly bounded differences (this bound also 
depends on $K_1$). Furthermore, as all $z_n\in K_1$ and $u_{T_n,P_n}(z_n)=0$, 
the sequence $u_{T_n,P_n}|_{\widetilde{V}_n\cap K_1}$ is uniformly bounded. We 
deduce from the Arzel\`{a}-Ascoli theorem that there is a continuous 
function $u_1:K_1\to\R$ and a subsequence $\{l_k\}$ with $l_1=n_1$ such that 
$u_{T_{l_k},P_{l_k}}$ converges to $u_1$ uniformly on $K_1$.

Analogously, we see that there is a continuous function $u_1:K_1\to\R$ and a 
subsequence $\{m_k\}$ of $\{l_k\}$  with $m_1=n_1$, $m_2=l_2$ such that 
$u_{T_{m_k},P_{m_k}}$ converges to $u_2$ uniformly on $K_2$. Clearly, we have 
$u_1=u_2$ on $K_1$. This procedure can be continued and eventually we obtain a 
continuous function $u:\widetilde{\cal R}\to\R$ and a subsequence $\{n_k\}$ of 
$\{1,2,3,\dots \}$ such that $u_{T_{n_k},P_{n_k}}$ converges uniformly to $u$ 
on each compact subset of $\widetilde{\cal R}$. Also, $u$ has the same periods 
$P$ as $u_{{\cal R}, P}$ and $u(z_0)=0$. Applying Lemma~\ref{lemHarmLimit} to
bounded domains not containing any branch point, we see that the limit function 
$u:\widetilde{\cal R}\to\R$ is harmonic in $\widetilde{\cal R}$ except possibly 
at the
branch points. But as $u$ is locally bounded, these singularities can be 
removed and therefore the continuous function $u$ is in fact harmonic on the 
whole surface $\widetilde{\cal R}$. Thus $u=u_{\widetilde{\cal R},P}$ by our 
normalization $u(z_0)=0=u_{\widetilde{\cal R},P}(z_0)$.

Since the limit function $u=u_{\widetilde{\cal R},P}$ is unique, it follows 
that the whole sequence $u_{T_n,P_n}$, not just the subsequence 
$u_{T_{n_k},P_{n_k}}$, converges to $u_{\widetilde{\cal R},P}$ uniformly on 
every compact subset.
\end{proof}

\begin{proof}[Proof of Theorem~\ref{theoConvAbelInt}]
 Let $P_n,P\in\R^{2g}$ be the periods of the real parts $\text{Re} 
\phi^l_{T_n}: \widetilde{V}\to\R$ and $\text{Re} \phi^l_{\cal R}: 
\widetilde{{\cal 
R}}\to \R$ of the discrete and smooth holomorphic integrals, 
respectively. Then by Theorem~\ref{theoAbelExist} $\text{Re} 
\phi^l_{T_n}=u_{T_n,P_n}$ and $\text{Re} 
\phi^l_{\cal R}= u_{\widetilde{\cal R},P}$. 
Theorem~\ref{theoPeriodConv}, implies that 
$P_n\to P$ as $n\to\infty$. Thus we deduce from 
Theorem~\ref{theoConvmultiharmonic} that the real parts $\text{Re} 
\phi^l_{T_n}$ converge to $\text{Re} \phi^l_{\cal R}$ uniformly on every 
compact subset. Convergence of the imaginary parts is proven analogously due to 
the following Lemma~\ref{lemConjFunc}.
\end{proof}

\begin{lemma}[Conjugate Functions Principle]\label{lemConjFunc}
 Let $f=(\text{Re} f :\widetilde{V}\to\R,\; \text{Im} f:\widetilde{F}\to\R)$ be 
a 
discrete holomorphic integral. Then $E_T(\text{Re} f)= E_T(\text{Im}f)$.
\end{lemma}
\begin{proof}
 This follows immediately from~\eqref{eqdefconj} together with the definitions 
of the discrete energies in~\eqref{eqdefE} and~\eqref{eqdefEconj}.
\end{proof}

\section{Improved convergences of period matrices and holomorphic integrals for 
polyhedral surfaces}\label{SecPoly}

The techniques applied for adapted triangulations near branch points may also 
be used to improve the order of convergence of period matrices and holomorphic 
integrals for polyhedral surfaces compared 
to the results obtained in~\cite{BoSk16}.
A polyhedral surface ${\cal S}$ is an oriented 
two-dimensional manifold without boundary which has a piecewise flat metric 
with isolated conical singularities. An example is the surface of a polyhedron 
in three-dimensional space. Let $T_{\cal S}$ be a geodesic triangulation of the 
polyhedral surface ${\cal S}$ such that all faces are flat triangles. Note in 
particular, that all singular points of the metric are vertices of $T_{\cal 
S}$. 
On all edges we use {cotan weights} given by~\eqref{eqcotandef}.

If $\gamma_O> 1/2$, we do not adapt the triangulation further. But for 
singularities $O$ with $\gamma_O\leq 1/2$ we consider a
chart $g_O$, which maps a neighborhood ${\cal C}_O$ of $O$ to a neighborhood of 
the origin in $\C$. Furthermore, we can introduce as above ``polar 
coordinates'' $(r,\phi)$ on ${\cal C}_O$ with the origin at the vertex $O$. We 
map all vertices in ${\cal C}_O$ to a neighborhood of the origin in $\C$ by the 
chart $g_O:{\cal C}_O\to\C$, 
$g_O(r,\phi):=r^{\gamma_O}\text{e}^{i\gamma_O\phi}$.
If $\gamma_O\leq 1/2$ we demand that the images of any two incident 
vertices in ${\cal C}_O$ have maximum distance $h$. 
Consider any triangle 
$\Delta$ in ${\cal C}_O$ whose vertex $z$ nearest to $O$ satisfies $|Oz|\geq 
h_O=h^{1/\gamma_O}$, where $|Oz|$ denotes the distance of $z$ to $O$ in $\cal 
S$. As in Section~\ref{subsecConv1} we deduce from our assumption that the 
maximal edge length in $\Delta$ 
is smaller than $h\cdot |Oz|^{1-\gamma_O}$.

Applying the estimates of Sections~\ref{SecEstTriangle} and~\ref{SecEstEnergy}, 
we obtain the following improved versions of 
Theorems~2.5 and~2.7 of~\cite{BoSk16}.

\begin{theorem}[Energy convergence]
 For each $\delta>0$ and each smooth multi-valued harmonic function 
$u:\widetilde{{\cal S}}\to \R$ there are two constants 
$\Const_{u,\delta,{\cal S}}, \const_{u,\delta,{\cal S}}>0$ such 
that for any 
adapted triangulation $T$ of ${\cal S}$ with maximal edge length 
$h<\const_{u,\delta,{\cal S}}$ and minimal face angle $>\delta$ we have
\[|E_T(u|_{\widetilde{V}})-E(u)|\leq \Const_{u,\delta,{\cal S}}\cdot h.\]
\end{theorem}

\begin{theorem}[Convergence of period matrices]
 For each $\delta>0$ there exist constants 
$\Const_{\delta,{\cal S}}$, $\const_{\delta,{\cal S}}>0$ such that 
for any adapted triangulation $T$ of ${\cal S}$ with maximal edge length 
$h<\const_{u,\delta,{\cal S}}$ and minimal face angle $>\delta$ we have
\[\|\Pi_T-\Pi_{\cal S}\|\leq \Const_{\delta,{\cal S}}\cdot h.\]
\end{theorem}

\begin{theorem}[Convergence of holomorphic integrals]
 Let $\{ T_n\}$ be a sequence of non-degenerate uniform adapted triangulations 
of ${\cal S}$ with maximal edge length $h_n\to 0$ as $n\to\infty$. Let 
$\pi:\widetilde{\cal S}\to S$ be the universal covering of ${\cal S}$. Denote 
by $\widetilde{T}_n$ the corresponding triangulation of $\widetilde{\cal S}$ 
such that 
$\pi(\widetilde{T}_n)=T_n$. Let $z_n\in 
\widetilde{V}_n$ be a sequence of vertices converging to a point 
$z_0\in\widetilde{\cal S}$. 
Let $\Delta_n\in \widetilde{F}_n$ be a sequence of faces with its vertices 
converging to $z_0$. Then for each $1\leq l\leq g$ the discrete holomorphic
integrals $\phi^l_T=(\text{Re} \phi^l_T :\widetilde{V}\to\R, \;
\text{Im} \phi^l_T:\widetilde{F}\to\R)$ normalized at $z_n$ and $w_n$ converge 
uniformly on each compact set to the holomorphic integral
$\phi^l_{\cal S}:\widetilde{{\cal S}}\to\C$ normalized at $z_0$.
\end{theorem}

\section{Numerical experiments}\label{secNum}

In the following, we present some numerical analysis for our convergence 
results detailed above. We are very grateful to Stefan Sechelmann for writing 
software and performing numerical experiments.

Mainly, we apply the scheme described in Section~\ref{SecResults}, but we 
consider 
the triangulations on the sphere $\Sp^2\cong \hC$ without stereographic 
projection to 
$\C$. Furthermore, we use an approximation of the discrete energy $E_T$ and of 
the discrete multi-valued harmonic functions $u_{T,P}$ because  we use slightly 
different 
weights instead of those given in Section~\ref{secDiscHarm}. In particular, 
each triangle $\Delta$ of $T$ as an embedded triangle in $\Sp^2$ before 
stereographic projection has circular 
arcs as edges where the circles all pass through the north pole ($=\infty$) if 
$\Delta$ is contained in the spherical region corresponding to $B_{\varrho}(0)$ 
or all pass through the south pole ($=0$) if $\Delta$ is contained in the 
spherical region corresponding to $\hC\setminus B_{\varrho}(0)$, respectively. 
For boundary triangles, there are two types of circular arcs. 
For practical reasons, we do not work with these triangles in 
$\Sp^2\subset\R^3$. Instead, we take the vertices and add straight line 
segments 
in $\R^3$ between incident vertices. For every original triangle 
$\Delta$ on $\Sp^2$ we obtain a corresponding triangle $\Delta^S$ in $\R^3$, 
see Figure~\ref{figExTriang} for some examples of triangulations.
\begin{figure}[tb]
\hfill
 \includegraphics[width=0.2\textwidth]{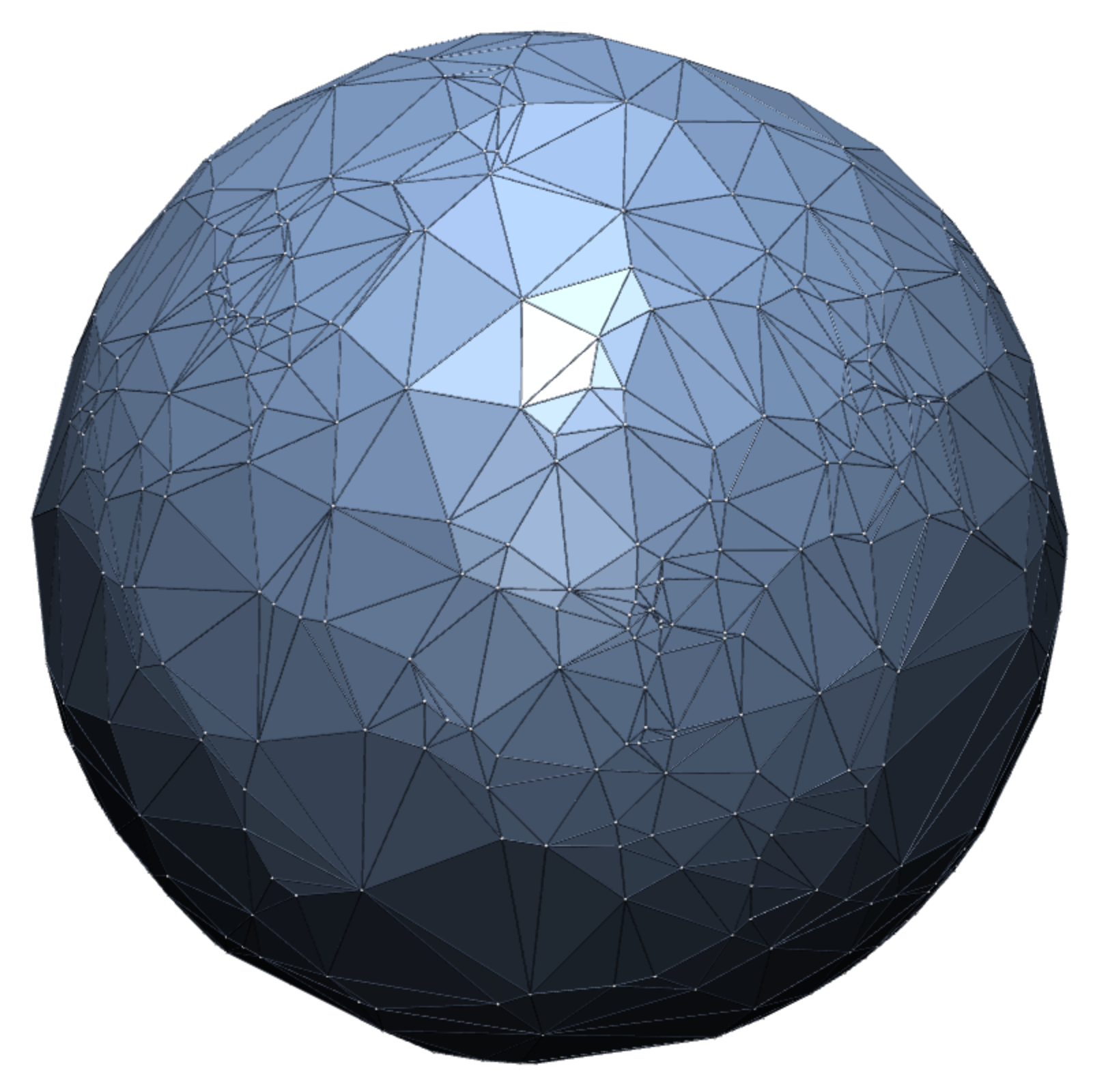}
\hfill
\includegraphics[width=0.2\textwidth]{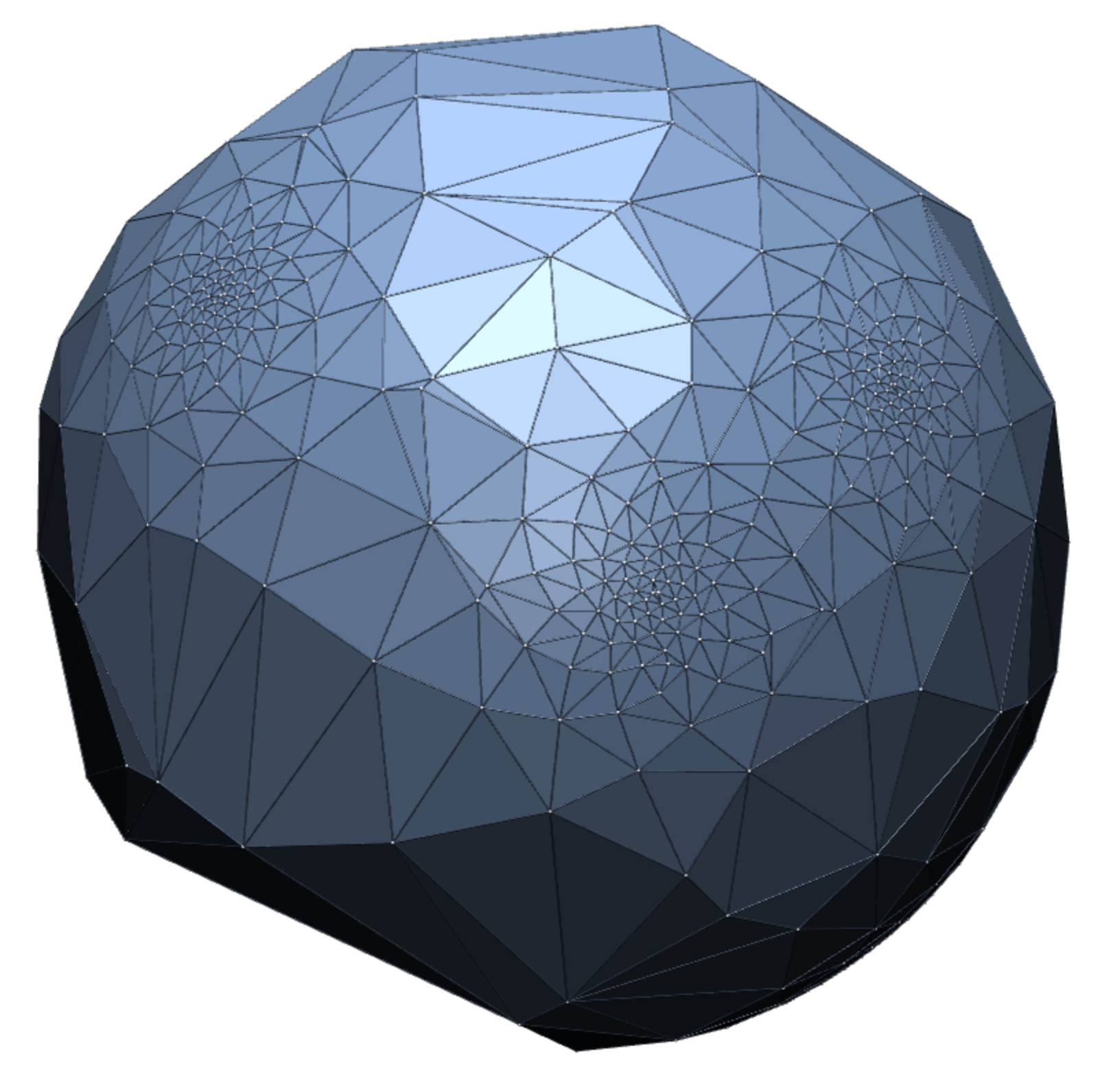}
\hfill
 \includegraphics[width=0.2\textwidth]{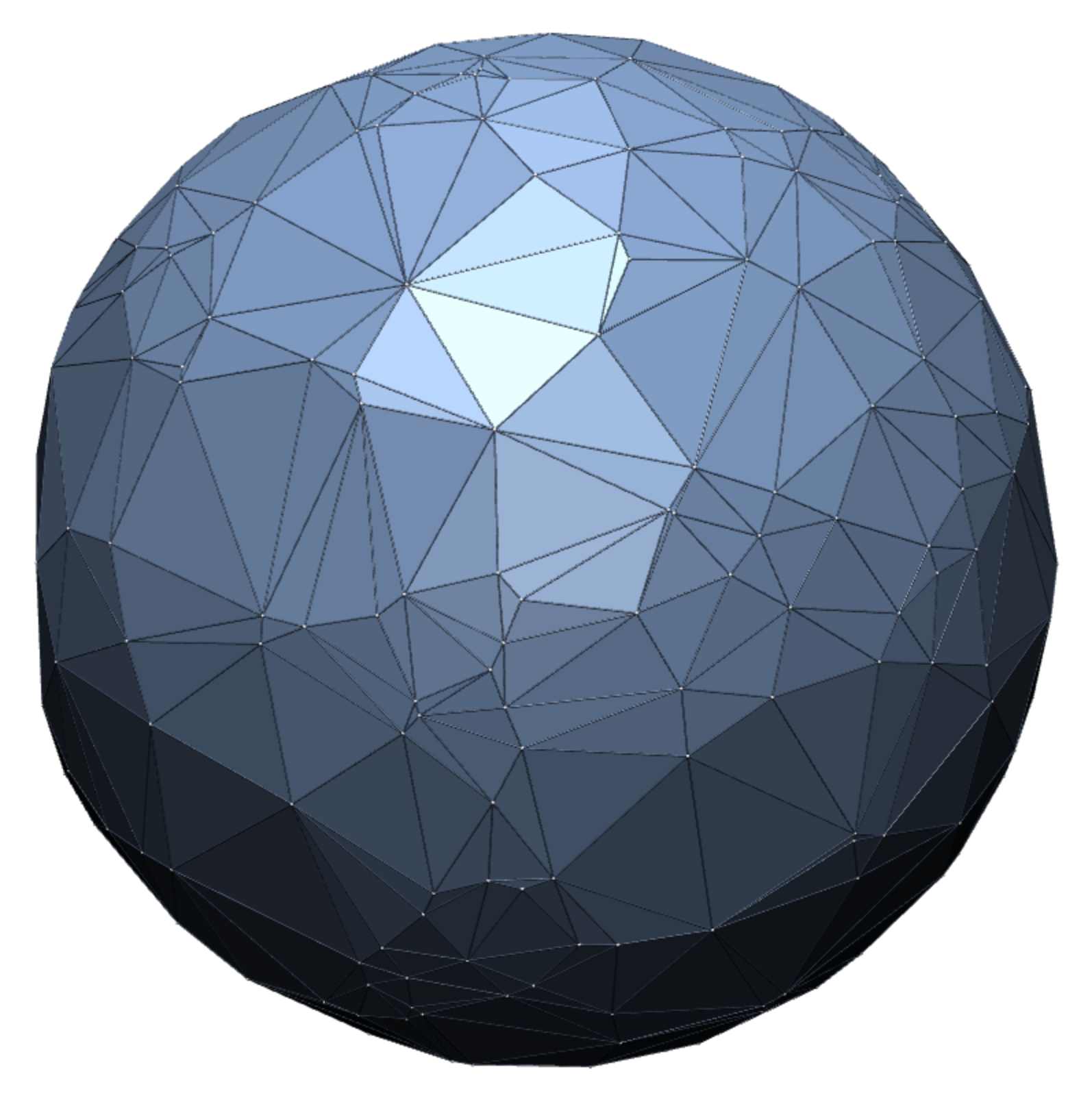}
\hfill
\includegraphics[width=0.2\textwidth]{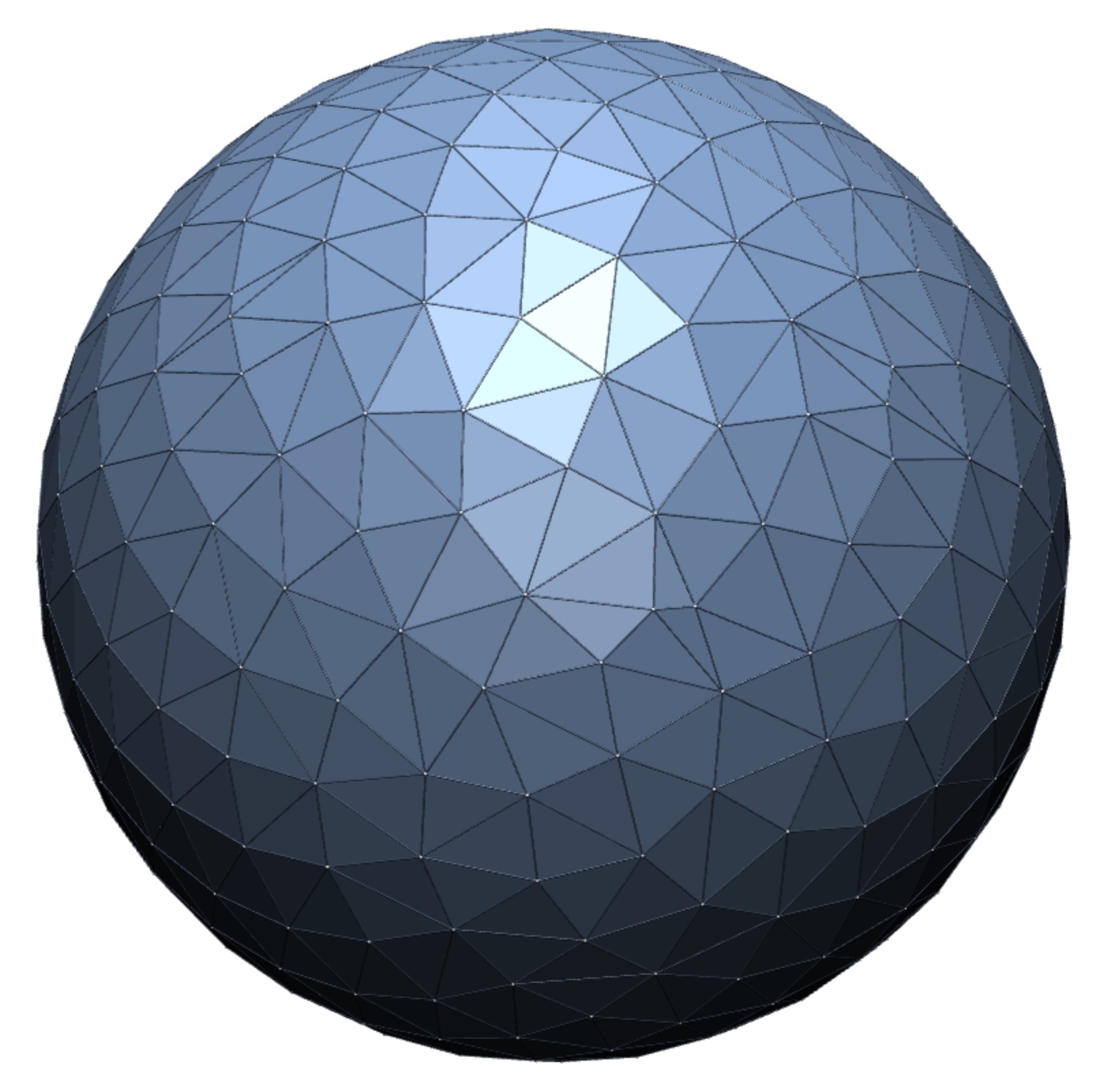}
\hfill
\caption{Examples of our four types of triangulations used for the torus $\cal 
T$: Random 
with adapted triangles (Clustering Random $\circ$), Fibonacci with adapted 
triangles (Clustering Fibonacci {\color{red}$\Box$}),
Random without adapted triangles (Homogeneous Random 
{\color{green}$\Diamond$}), Fibonacci without adapted 
triangles (Homogeneous Fibonacci 
{\color{blue}$\bigtriangleup$})}\label{figExTriang}
\end{figure}
As $\varrho$ is fixed and the 
maximum edge length $h$ tends to zero, any angle $\alpha$ in the triangle 
$\Delta$ and the corresponding angle $\alpha^S$ in the triangle $\Delta^S$ 
only differ by an error of order $h$, in particular $|\alpha-\alpha^S|\leq 
\Const_{\delta,\rho}\cdot l_{\max}(\Delta)$. Thus, for uniform Delaunay 
triangulations which we consider, the weights $c^S(e)= \frac{1}{2}(\cot 
\alpha_e^S +\cot \beta_e^S)$, using the cotan-formula for the angles 
of the triangles $\Delta^S$, can be estimated using the original weights $c(e)$ 
for $T$. More precisely, there is a constant $\Const_{\delta,\varrho}$ such that
$c(e)\cdot (1-\Const_{\delta,\varrho}\cdot h) \leq c^S(e)\leq c(e)\cdot 
(1+\Const_{\delta,\varrho}\cdot h)$. This implies $E_T(u)\cdot 
(1-\Const_{\delta,\varrho}\cdot h)\leq E^S_T(u)\leq E_T(u) \cdot 
(1+\Const_{\delta,\varrho}\cdot h)$. Thus for $0<h<1/\Const_{\delta,\varrho}$ 
and some constant $\Const_{P,\delta,{\cal R}, \varrho}>0$ we 
obtain similarly as in the proof of Lemma~\ref{lemEnergyFormConv} that 
\[ |E_T^S(u^S_{T,P})-E(u_{{\cal R},P})| \leq \Const_{P,\delta,{\cal R}, 
\varrho}\cdot h.\]

As concrete examples we consider two surfaces with known period matrices, 
namely the torus $\cal T$ of genus $1$ with branch points $0.5+0.4i$, 
$-0.3+0.2i$, $-0.1$, $0.1-0.2i$ and Lawson's minimal surface~$\cal L$ of genus 
$2$ which corresponds to the hyperelliptic curve $\mu^2=\lambda^6-1$ with 
branch points $\text{e}^{ik\pi/3}$, $k=0,1,\dots,5$. The smooth period matrices 
are $\Pi_{\cal T}\approx0.836+0.955i$ and $\Pi_{\cal 
L}=\frac{i}{\sqrt{3}}\left(\begin{smallmatrix} 2 
& -1 \\ -1 & 2 \end{smallmatrix}\right)$.


We compare four different types of triangulations of $\Sp^2$ which are used as 
basis for the computations of the discrete period matrices. 
Figure~\ref{figExTriang} shows an example for each of these four types. 
In order to simplify calculations and the construction of 
cycles, we always use the same triangulation on every sheet of the covering. 
\begin{description}
\item[Random] We sample points at random on the sphere and then build 
the corresponding Delaunay triangulation. 
\item[Fibonacci] The points on a sphere are evenly distributed by means of 
a Fibonacci spiral. This leads to very 'regular' triangulations with 
triangles which are almost equilateral, except near branch points, which are in 
general additional vertices. 
\end{description}
These two types of triangulations are directly used for further 
computations (called Homogeneous Random {\color{green}$\Diamond$} and 
Homogeneous Fibonacci {\color{blue}$\bigtriangleup$}).
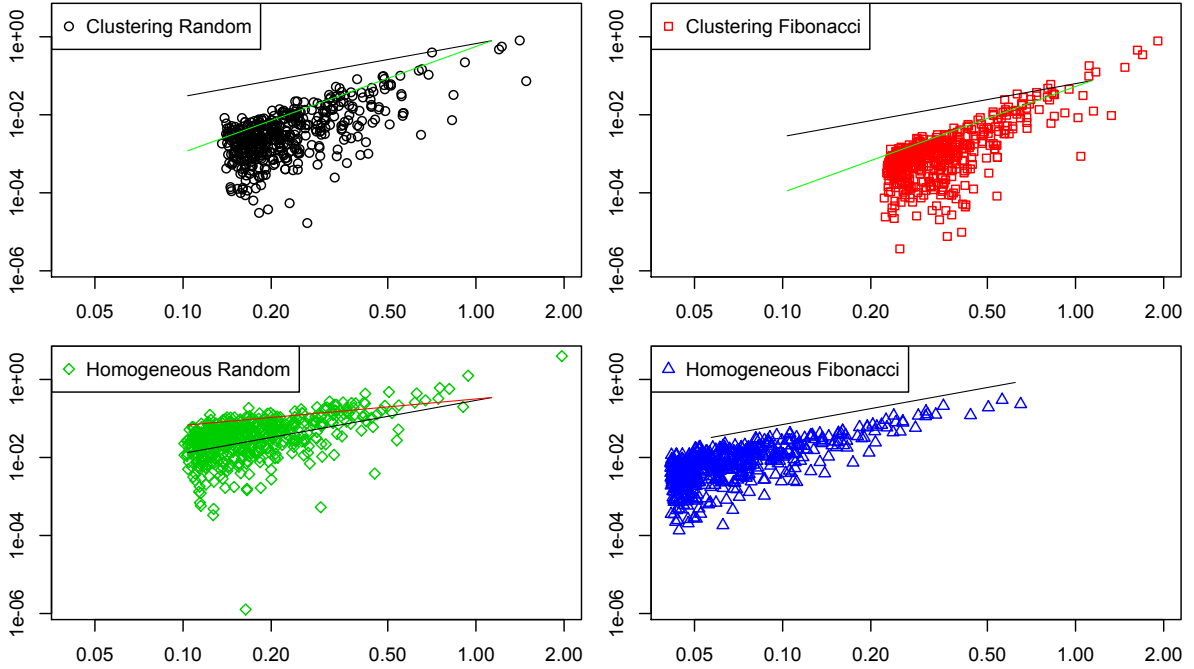
\begin{figure}[tb]
\begin{tikzpicture}
\node at (-2,0){\includegraphics[width=\textwidth]{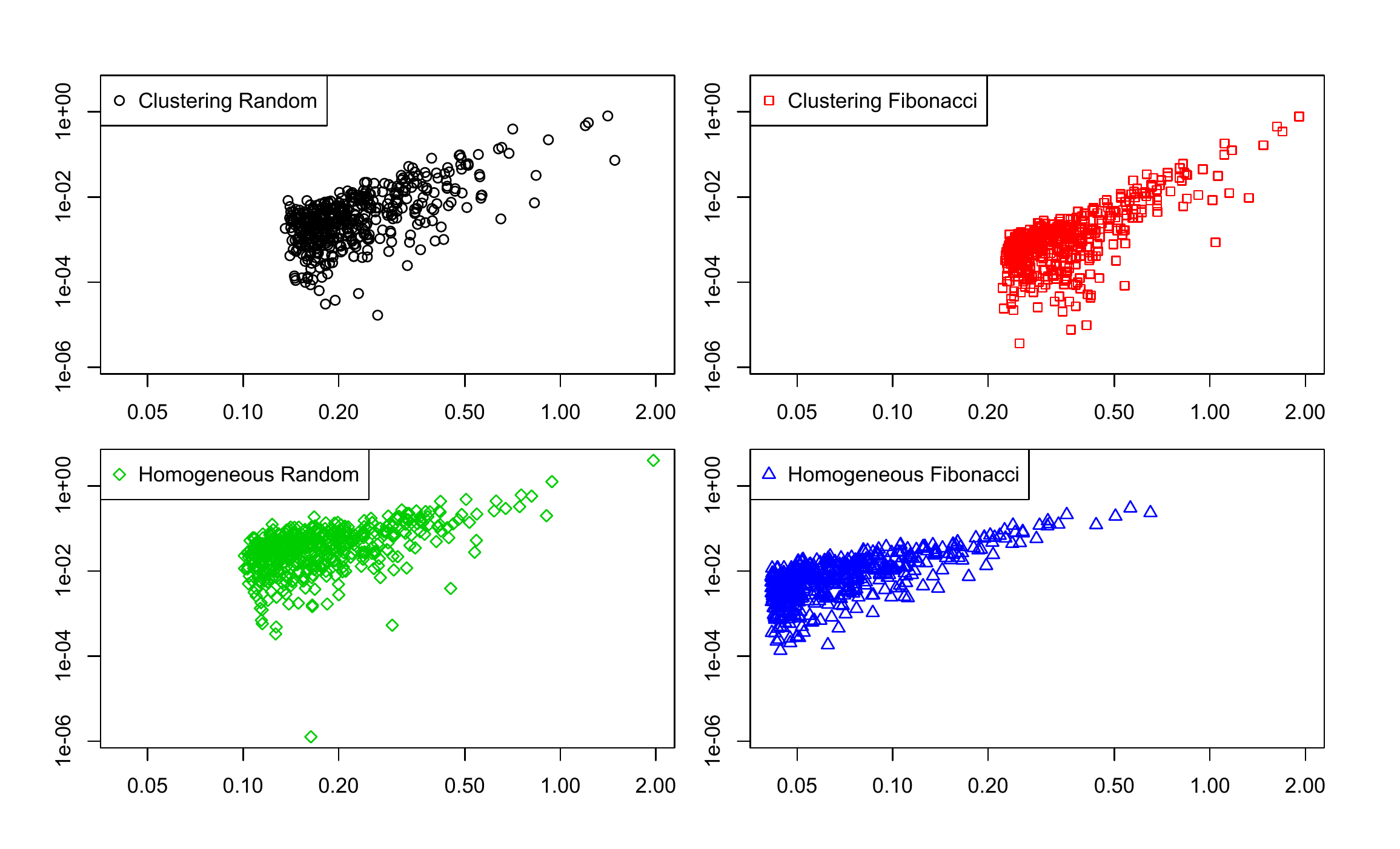}};
\draw (-7.492,3.012) -- (-3.482,3.743);
\draw[color=green] (-7.492,2.281) -- (-3.482,3.743);
\draw (-7.492,-1.719) -- (-3.482,-0.988);
\draw[color=red] (-7.492,-1.3535) -- (-3.482,-0.988);

\draw (0.397,2.481) -- (4.407,3.212);
\draw[color=green] (0.397,1.75) -- (4.407,3.212);
\draw (-0.603,-1.519) -- (3.407,-0.788);
\end{tikzpicture}
\caption{Scattering-plot of the approximation error for the period 
matrices for examples with different maximal edge length for the four 
different types of triangulations of the torus $\cal T$. The black line (in all 
plots) has slope $1$, the green line (in 
the two Clustering plots) has slope~$2$ and the red line (in the Homogeneous 
Random plot) has slope $1/2$. 
}\label{figTorus}
\end{figure}
\begin{figure}[tb]
\begin{tikzpicture}
\node at (0,0){\includegraphics[width=\textwidth]{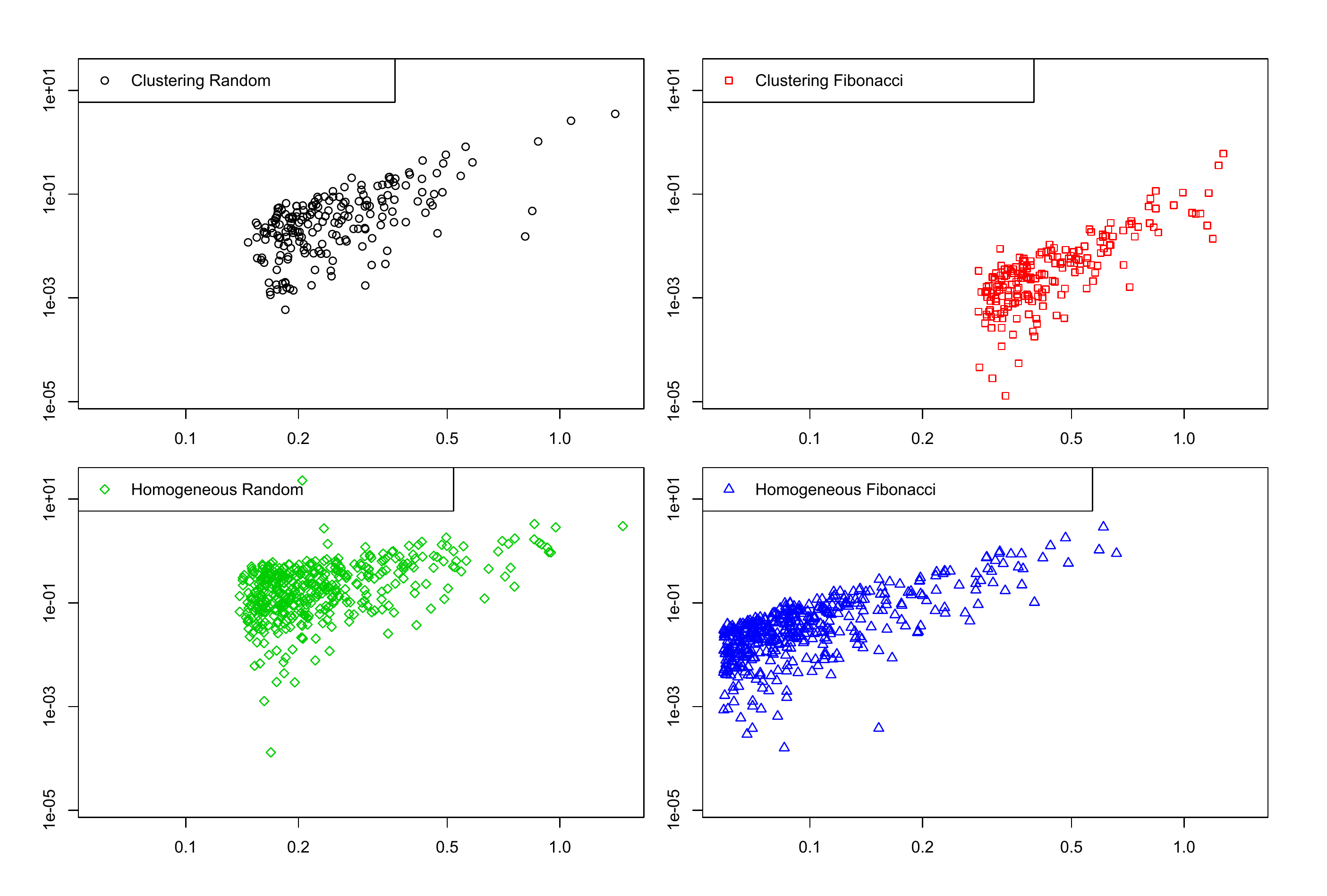}};
\draw (-6.12,3.287) -- (-1.327,3.952);
\draw (-6.12,-2.213) -- (-1.327,-1.548);
\draw[color=red] (-6.12,-1.8) -- (-1.327,-1.4675);
\draw[color=green] (-6.12,2.622) -- (-1.327,3.952);

\draw (1.88,2.587) -- (6.673,3.252);
\draw (1.88,-1.713) -- (6.673,-1.048);
\draw[color=green] (1.88,1.922) -- (6.673,3.252);
\end{tikzpicture}
\caption{Scattering-plot of the approximation error for the period 
matrices for examples with different maximal edge length for the four 
different types of triangulations of Lawson's minimal surface $\cal L$ 
of genus $2$. The black line (in all plots) has slope $1$, the green line (in 
the two Clustering plots) has slope~$2$ and the red line (in the Homogeneous 
Random plot) has slope $1/2$. 
}\label{figLawson}
\end{figure}
In these two cases the estimates of~\cite{BoSk16} apply. Our numerical results 
are plotted in the lower rows of Figures~\ref{figTorus} and~\ref{figLawson}. 
The log-log plots show that the error behaves 
indeed like $\sqrt{h}$ for 'Homogeneous Random', where $h$ is the maximal edge 
length. This was also observed in the example studied in Section~7.3 
of~\cite{BoSk16}. 
Nevertheless, for the more regular triangulations using Fibonacci spirals, our 
numerical evidence indicates a higher order of the error bound, possibly a 
linear dependence on $h$.

According to our new idea of adapted triangulations explained 
in Section~\ref{subsecConv1}, we refine the examples of the two types of 
triangulations above in a neighborhood of the branch points by suitably adding 
vertices (called Clustering Random $\circ$ and Clustering Fibonacci 
{\color{red}$\Box$}). Our results in Figures~\ref{figTorus} and~\ref{figLawson} 
confirm that the error between smooth and discrete period matrices 
decreases indeed faster in the adapted case. In particular, the log-log plots 
in the upper row of Figures~\ref{figTorus} and~\ref{figLawson}, respectively, 
show that for our adapted method the 
error depends at least linearly on the maximal edge length $h$ (as proven in 
Theorem~\ref{theoPeriodConv}) and is again possibly even of higher order for 
more regular triangulations using Fibonacci spirals.
Recall that the actual bounds on the approximation error depend on the angles 
in the triangles which differ for all our examples. In our proof we only 
use some (rough) estimate of the angles such that our constants in 
Theorems~\ref{lemEnergyFormConv} and~\ref{theoPeriodConv} depend only on the 
minimal angle of the adapted triangulation. We did not study 
the dependence of the angles in detail, but our numerical results suggest that 
the order of convergence also depends significantly on the regularity of the 
triangulations.


\subsection*{Acknowledgments}
The authors especially thank Stefan 
Sechelmann for writing software and creating examples for numerical 
experiments. 

This research was supported by the DFG Collaborative Research Center
TRR~109 ``Discretization in Geometry and Dynamics''.

\appendix
\section{Appendix}\label{appendix}

\subsection{Interpolation function on boundary triangles and estimates on 
corresponding edge weights}\label{secApp1}

In the following, we expose the calculations for the energy of the 
interpolation function and the corresponding edge weights.

Let $\Delta[x,y,z]\in F_{\varrho}$ be a boundary triangle. Without loss of 
generality, we assume that the vertices are labelled such that $x\in 
B_{\varrho}$ and $y,z\in \C\setminus B_{\varrho}$. Therefore, the triangle 
$\Delta[x,y,z]$ is bounded by two straight edges $[x,y]$ and $[x,z]$ and by the 
trace of the curve $s:[0,1]\to\C$, $s(t)=\frac{yz}{z+t(y-z)}$, connecting $y$ 
and $z$ which is in 
general a circular arc. We parametrize this triangle by
\begin{equation*}
 p:[0,1]\times[0,1]\to\Delta[x,y,z],\quad p(\tau,\sigma)= x+\sigma(s(\tau)-x).
\end{equation*}
Note that $p$ is bijective for $\sigma\not=0$. In this parametrization, the 
interpolation function is $u_\Delta(\tau,\sigma):= I_Tu(p(\tau,\sigma)) 
=u_x+\sigma(u_y-u_x +\tau(u_z-u_y))$ as explained in 
Section~\ref{secDiscHarm}. Here we use the notation $u_v=u(v)$ for the values 
of the given smooth function $u$ at the vertices $x,y,z$.
Therefore, we obtain
\begin{align}
 \int\limits_{\Delta[x,y,z]}|\nabla I_Tu|^2 &= \int_0^1\int_0^1 
Du_\Delta(\tau,\sigma) Dp^{-1}(p(\tau,\sigma))(Dp^{-1}(p(\tau,\sigma)))^T 
(Du_\Delta(\tau,\sigma))^T |\det Dp(\tau,\sigma)| d\tau d\sigma \notag \\
&= \int_0^1\int_0^1 \left( (u_z-u_y)^2 \frac{|s(\tau)-x|^2 
-2\tau\text{Re}((s(\tau)-x) \overline{s'(\tau)}) 
+\tau^2|s'(\tau)|^2}{|\text{Im}((s(\tau)-x) \overline{s'(\tau)})|^2}\right. 
\notag\\
&\qquad\qquad\quad + (u_y-u_x)^2 \frac{|s'(\tau)|^2}{|\text{Im}((s(\tau)-x) 
\overline{s'(\tau)})|^2} \label{eqestFrho}\\ 
&\qquad\qquad\quad \left.+ (u_y-u_x)(u_z-u_y) \frac{2\tau|s'(\tau)|^2 
-2\text{Re}((s(\tau)-x) 
\overline{s'(\tau)})}{|\text{Im}((s(\tau)-x) \overline{s'(\tau)})|^2}\right)  
|\det Dp(\tau,\sigma)| d\tau d\sigma \notag\\
&=\frac{1}{2} \int_0^1 \left( (u_z-u_y)^2 \frac{|s(\tau)-x|^2 
-2\tau\text{Re}((s(\tau)-x) \overline{s'(\tau)})
+\tau^2|s'(\tau)|^2}{|\text{Im}((s(\tau)-x) \overline{s'(\tau)})|}  \right. 
\notag\\
&\qquad\qquad\quad + (u_y-u_x)^2 \frac{|s'(\tau)|^2}{|\text{Im}((s(\tau)-x) 
\overline{s'(\tau)})|}  \notag\\
&\qquad\qquad\quad \left.+ (u_y-u_x)(u_z-u_y) \frac{2\tau|s'(\tau)|^2 
-2\text{Re}((s(\tau)-x) \overline{s'(\tau)})}{|\text{Im}((s(\tau)-x) 
\overline{s'(\tau)})|}\right) d\tau,\notag
\end{align}
as $Du_\Delta(\tau,\sigma)= (\sigma(u_z-u_y), u_y-u_x +\tau(u_z-u_y))$, 
$Dp(\tau,\sigma)= \begin{pmatrix} \sigma \text{Re}(s'(\tau)) & 
\text{Re}(s(\tau)-x) \\ \sigma \text{Im}(s'(\tau)) & 
\text{Im}(s(\tau)-x) \end{pmatrix}$, and $\det(Dp(\tau,\sigma))= \sigma
\text{Im}((s(\tau)-x) \overline{s'(\tau)})$. Thus we deduce that
\begin{align}
\int\limits_{\Delta[x,y,z]} |\nabla I_Tu|^2 &= C_{[x,y]}(u(x)-u(y))^2 + 
C_{[y,z]}(u(y)-u(z))^2 +C_{[z,x]}(u(z)-u(x))^2, \notag \\
\text{where}\qquad  C_{[x,y]}&=\frac{1}{2} \int_0^1 
\frac{(1-\tau)|s'(\tau)|^2 +\text{Re}((s(\tau)-x))
\overline{s'(\tau)})}{|\text{Im}((s(\tau)-x) \overline{s'(\tau)})|} d\tau, 
\label{eqCxy} \\
C_{[y,z]} &= \frac{1}{2} \int_0^1 \frac{|s(\tau)-x|^2+ 
\tau(\tau-1)|s'(\tau)|^2 +\text{Re}((s(\tau)-x) 
\overline{s'(\tau)})}{|\text{Im}((s(\tau)-x) \overline{s'(\tau)})|} 
d\tau,\label{eqCyz} \\
C_{[z,x]}&=\frac{1}{2} \int_0^1 \frac{\tau |s'(\tau)|^2 -\text{Re}((s(\tau)-x)
\overline{s'(\tau)})}{|\text{Im}((s(\tau)-x) \overline{s'(\tau)})|} d\tau. 
\label{eqCzx}
\end{align}
This gives an explicit way to calculate the edge weights. Note that by the same 
method we can obtain the usual cotan-weights on the Euclidean triangle with  
vertices $x,y,z$, if we use $s_E(t)=y+t(z-y)$ instead of $s(t)$ for 
$t\in[0,1]$. The function $s_E$ is the usual linear parametrization of the 
straight edge from $y$ to $z$.

An important observation is that these seemingly 'complicated' weights are in 
fact only 
small perturbations of the usual cotan-weights if the edge length is small 
enough. To see this, we will estimate the quantities in the above integrals 
compared to the corresponding quantities for $s_E$. 

\begin{proof}[Proof of the estimate in Remark~\ref{remest}]
We assume that there is some $\delta>0$ such that all angles in the triangle 
$\Delta[x,y,z]$ are in $[\delta,\pi-\delta]$ and also all angles of the 
Euclidean triangle with vertices $x,y,z$ are in $[\delta,\pi-\delta]$.
In the following, we will always assume that $\tau\in[0,1]$ as in the integral 
terms above. Also, we are not interested in the best possible estimates, any 
constant, depending only on the indicated parameter, will suffice.

First note that $|\text{Im}((s(\tau)-x) \overline{s'(\tau)})|\geq 
|s(\tau)-x||s'(\tau)|\sin\delta$ and $|\text{Im}((s_E(\tau)-x) 
\overline{s_E'(\tau)})|\geq |s_E(\tau)-x||z-y|\sin\delta$. Denote by $h$ the 
maximal edge length of $\Delta[x,y,z]$, so $h\geq 
\max\{|y-x|,|z-x|,\text{length}(s)\}$. As $\Delta[x,y,z]$ is a boundary 
triangle, we have 
\[ \frac{\varrho^2}{(\varrho+h)^2}\leq \left|\frac{s'(\tau)}{s_E(\tau)}\right|= 
\frac{|s'(\tau)|}{|y-z|} \leq \frac{(\varrho+h)^2}{\varrho^2}.\]
Furthermore, we deduce that $1-\Const_{\delta,\varrho}\cdot h\leq 
\left|\frac{s(\tau)-x}{s_E(\tau)-x}\right|\leq 1+\Const_{\delta,\varrho}\cdot 
h$ as
\[\left|\frac{s(\tau)-s_E(\tau)}{s_E(\tau)-x}\right| = 
\left|\frac{\tau(\tau-1)(y-z)^2}{z+\tau(y-z)}\cdot 
\frac{1}{(y-x)+\tau(z-y)}\right| \leq \frac{h}{\varrho\sin^2\delta}.\]
Further, note that using the sine law
\[\sin^2\delta \leq \frac{|s_E(\tau)-x|}{|s_E'(\tau)|} 
=\frac{|(y-x)+\tau(z-y)|}{|y-z|} \leq 1+\frac{1}{\sin\delta}.\]
Combining these estimates, we have 
\begin{align*}
\frac{|s(\tau)-s_E(\tau)| 
|s'(\tau)|}{|\text{Im}((s_E(\tau)-x) \overline{s_E'(\tau)})|} &\leq 
\frac{|s(\tau)-s_E(\tau)|}{|s_E(\tau)-x|} \cdot 
\left|\frac{s'(\tau)}{s_E(\tau)}\right| \cdot\frac{1}{\sin\delta} \leq 
\Const_{\delta,\varrho}\cdot  h,\\
\frac{|s'(\tau)|^2}{|\text{Im}((s_E(\tau)-x) \overline{s_E'(\tau)})|} &\leq 
\left|\frac{s'(\tau)}{s_E(\tau)}\right| \cdot 
\frac{|s_E'(\tau)|^2}{|s_E'(\tau)| |s_E(\tau)-x|\sin\delta} \leq 
\Const_{\delta,\varrho}\cdot  h,\\
\frac{|s(\tau)-x|^2}{|\text{Im}((s_E(\tau)-x) \overline{s_E'(\tau)})|} &\leq 
\frac{|s(\tau)-x|^2}{|s_E(\tau)-x|^2}\cdot \frac{|s_E(\tau)-x|^2}{|s_E'(\tau)| 
|s_E(\tau)-x|\sin\delta} \leq \Const_{\delta,\varrho}\cdot  h.
\end{align*} 
Furthermore, as $h\leq \varrho/2$,
\[ \frac{|s_E(\tau)-x| |s'(\tau)-s_E'(\tau)|}{|\text{Im}((s_E(\tau)-x) 
\overline{s_E'(\tau)})|} \leq \frac{|s'(\tau)-s_E'(\tau)|}{|y-z|\sin\delta} = 
\frac{|y-z||\tau^2(y-z)^2+2\tau z-z|}{|z+\tau(y-z)|^2\sin\delta} \leq 
\frac{6}{\varrho\sin\delta}\cdot h.\]
Therefore, we obtain
\begin{multline*}
\left| \frac{\text{Im}((s(\tau)-x) 
\overline{s'(\tau)})}{\text{Im}((s_E(\tau)-x) \overline{s_E'(\tau)})} -1\right|
\\ 
\shoveright{ = \left| \frac{\text{Im}((s(\tau)-s_E(\tau)) 
\overline{s'(\tau)})}{\text{Im}((s_E(\tau)-x) \overline{s_E'(\tau)})} + 
\frac{\text{Im}((s(\tau)-x) 
\overline{(s'(\tau)-s_E'(\tau))})}{\text{Im}((s_E(\tau)-x) 
\overline{s_E'(\tau)})}\right| \leq \Const_{\varrho,\delta}\cdot h} \\
\shoveleft{ \left| \frac{\text{Re}((s(\tau)-x) 
\overline{s'(\tau)})- \text{Re}((s_E(\tau)-x) 
\overline{s_E'(\tau)})}{\text{Im}((s_E(\tau)-x) \overline{s_E'(\tau)})}\right|}
\\
=\left| \frac{\text{Re}((s(\tau)-s_E(\tau)) 
\overline{s'(\tau)})}{\text{Im}((s_E(\tau)-x) \overline{s_E'(\tau)})} + 
\frac{\text{Re}((s_E(\tau)-x) 
\overline{(s'(\tau)-s_E'(\tau))})}{\text{Im}((s_E(\tau)-x) 
\overline{s_E'(\tau)})}\right| 
\leq \Const_{\varrho,\delta}'\cdot h 
\end{multline*}
This also implies $\left| \frac{\text{Re}((s(\tau)-x) 
\overline{s'(\tau)})}{\text{Im}((s_E(\tau)-x) \overline{s_E'(\tau)})}\right| 
\leq \left| \frac{\text{Re}((s_E(\tau)-x) 
\overline{s_E'(\tau)})}{\text{Im}((s_E(\tau)-x) \overline{s_E'(\tau)})}\right| 
+ \Const_{\varrho,\delta}'\cdot h
\leq \Const_{\varrho,\delta}''$ 
Denote by $\alpha^E_z\in[\delta,\pi-\delta]$ the angle in the Euclidean 
triangle with vertices $x,y,z$. Then the previous estimates imply
\begin{align*}
 |C_{[x,y]}-\frac{1}{2}\cot\alpha_z^E| &\leq \frac{1}{2} \int_0^1 \left(
\frac{(1-\tau)(|s'(\tau)|^2\left|\frac{\text{Im}((s(\tau)-x) 
\overline{s'(\tau)})}{\text{Im}((s_E(\tau)-x) \overline{s_E'(\tau)})} 
-1\right|+\left|\frac{|s'(\tau)|^2}{|s_E'(\tau)|^2} -1\right| 
|s_E'(\tau)|^2)}{|\text{Im} ((s_E(\tau)-x) \overline{s_E'(\tau)})|}\right. \\
&\qquad \qquad + \frac{|\text{Re}((s(\tau)-x) 
\overline{s'(\tau)})| \left|\frac{\text{Im}((s(\tau)-x) 
\overline{s'(\tau)})}{\text{Im}((s_E(\tau)-x) \overline{s_E'(\tau)})} 
-1\right|}{|\text{Im}((s_E(\tau)-x) \overline{s_E'(\tau)})|}\\
&\qquad\qquad \left.+ 
\frac{|\text{Re}((s(\tau)-x) 
\overline{s'(\tau)})- \text{Re}((s_E(\tau)-x) 
\overline{s_E'(\tau)})|}{|\text{Im}((s_E(\tau)-x) \overline{s_E'(\tau)})|} 
\right) d\tau \\
& \leq \Const_{\varrho,\delta}'''\cdot h.
\end{align*}
Similarly, we can deduce that $|C_{[y,z]}-\frac{1}{2}\cot\alpha_x^E| \leq 
\Const_{\varrho,\delta}'''\cdot h$ and $|C_{[z,x]}-\frac{1}{2}\cot\alpha_y^E| 
\leq \Const_{\varrho,\delta}'''\cdot h$.
\end{proof}

\subsection{Proof of Lemma~\ref{lemEstF0}}\label{app2}

 Note that as $u$ is smooth, there exist constants 
$\Const_{u,\varrho},\const_{u,\varrho}>0$ such that for 
$0<h<\const_{u,\varrho}$ 
we have $|u(x)-u(y)|^2/|x-y|^2\leq 
\Const_{\delta,\varrho}\cdot\max\limits_{w\in\Delta}\|D^1u(w)\|^2 \leq 
\Const_{u,\varrho}$ for all edges $e=[x,y]$ of boundary triangles in $F_\varrho$ 
with edge lengths smaller than $h$.

Using our estimates of Section~\ref{secApp1} we deduce that there exists a 
constant $\Const_{\delta,\varrho}$ such that under our assumptions on angles 
and edge lengths we have the following estimates: For every boundary triangle 
$\Delta[x,y,z]\in F_{\varrho}$ and with the notation of Section~\ref{secApp1}
\begin{align*}
 &|z-y|^2 \frac{|s(\tau)-x|^2}{|\text{Im}((s(\tau)-x) 
\overline{s'(\tau)})|^2}\leq \Const_{\delta,\varrho}, \\
& |y-x|^2 \frac{|s'(\tau)|^2}{|\text{Im}((s(\tau)-x) \overline{s'(\tau)})|^2} 
\leq \Const_{\delta,\varrho} ,\\ 
&|y-x||z-y| \frac{|\text{Re}((s(\tau)-x) 
\overline{s'(\tau)})|}{|\text{Im}((s(\tau)-x) \overline{s'(\tau)})|^2}\leq 
\Const_{\delta,\varrho}.
\end{align*}
Now formula~\eqref{eqestFrho} leads to the estimate
\begin{equation}\label{eqestEDelta}
\int\limits_{\Delta[x,y,z]}|\nabla I_Tu|^2 \leq \Const_{\delta,\varrho,u}\cdot 
\text{Area}(\Delta[x,y,z]),
\end{equation}
where $\Const_{\delta,\varrho,u} \leq \Const_{\delta,\varrho}\cdot 
\max\limits_{w\in\Delta}\|D^1u(w)\|^2$.
Summing up these energies, we obtain
\begin{align*}
 E_{F_{\varrho}}(u) &=\sum_{\Delta\in F_{\varrho}} E_{T_\Delta}(u) 
=\sum_{\Delta\in F_{\varrho}} \int\limits_{\Delta[x,y,z]}|\nabla I_Tu|^2 
\leq \sum_{\Delta\in F_{\varrho}} \Const_{\delta,\varrho,u} \cdot
\text{Area}(\Delta[x,y,z]) \\
&\leq \Const_{\delta,\varrho,u}\cdot \text{Area}( B_{\varrho+h}(0)\setminus 
B_{\varrho-h}(0)) 
\leq \Const_{\delta,\varrho,u} \cdot 4h\varrho \cdot d \leq
\Const_{u,\delta,\varrho,{\cal R}}\cdot h,
\end{align*}
where $d$ denotes the degree of the covering map for $\cal R$.

\subsection{Proof of Equicontinuity Lemma~\ref{lemEqui}}\label{app3}

First note that condition~(D) from Section~\ref{SecConvAbel} implies that for 
every path in a non-degenerate uniform adapted triangulation $T$ with 
consecutive vertices $v_0v_1\dots v_m$ we have
\begin{align}\label{eqestPath}
E_{v_0v_1\dots v_m}(u):=& \sum_{k=1}^m c([v_{k-1},v_k])\cdot 
(u(v_k)-u(v_{k-1}))^2 \geq \sum_{k=1}^m \frac{\Const}{e}\cdot 
(u(v_k)-u(v_{k-1}))^2 
\notag \\ &\geq \frac{\Const}{e}\cdot \frac{1}{m} \left(\sum_{k=1}^m 
u(v_k)-u(v_{k-1})\right)^2=\frac{\Const}{e\cdot m} (u(v_m)-u(v_0))^2.
\end{align}
Here $e$ denotes the eccentricity as defined in Section~\ref{SecEqui} and for the last estimate we have used Schwarz's inequality.

Now consider a simply connected triangulation $T'$ with boundary contained in 
an open disc $B_r(v)\subset\C$. This is the assumption for part~(i) of the 
lemma. In the case of part~(ii), we consider the image triangulation $T_g'$ by 
the chart $g_O(z)=(z-O)^{\gamma_O}$. By abuse of notation, we still denote this image triangulation
by $T'$. Also, we denote the vertices of $T'$ by $Z$ and $W$, which are the 
actual vertices $z$, $w$ in the first case and the images $Z=g_O(z)$, $W=g_O(w)$ 
in the case of a branch point. For simplicity, we assume that the edges between 
vertices are straight line segments, as we do not need the actual, possibly 
curved edges. 

Let $u:V'\to\R$ be any function which assumes its maximum and its 
minimum on the boundary for any subgraph of $T'$. Let $Z$, $W$ be two distinct 
interior vertices of $T'$. Denote by $ZW$ the straight line segment joining 
these points. Let $dist(ZW,\partial T')$ be the Euclidean distance of this 
straight line segment to the curve of boundary edges. We assume that 
$|Z-W|<r/3<dist(ZW,\partial T')/3$ for some $r>0$. Let $h'$ denote twice the 
maximum circumradius of the triangles of $T'$. Let $m=\lfloor 
\frac{r-|Z-W|}{2h'}\rfloor$ be the largest integer smaller than 
$\frac{r-|Z-W|}{2h'}$. We consider auxiliary rectangles $R_k$, $k=1,\dots,m$, 
which are centered at $(Z+W)/2$ with one pair of sides parallel to $ZW$ with 
length $ZW+2k\cdot h'$ and other pair of sides orthogonal to $ZW$ of length 
$2k\cdot h'$. Then the interior of $R_k$, $k=1,\dots,m$, is covered by triangles 
of $T'$. Denote by $V_k'$ the set of vertices contained in $R_k\setminus 
R_{k-1}$, where $R_0=ZW$. Then any two vertices $v_A,v_B\in V'_k$ may be 
connected by a path $v_0v_1\dots v_N$ with $v_0=v_A$, $v_N=v_B$ and all 
vertices $v_j\in V'_k$ as $h'$ is larger than any edge length. 

Without loss of generality, assume that $u(Z)\geq u(W)$. As $u$ assumes its 
maximum and minimum on the boundary, there exists $Z_k,W_k\in V'_k$ such that 
$u(Z_k)\geq u(Z)\geq u(W)\geq u(W_k)$. The length of the path joining $Z_k$ and 
$W_k$ is at most the number of vertices in $V'_k$. The set of these vertices 
can be covered by at most $\Const\cdot (|Z-W|+4kh')/h'$ discs of radius $h'/2$. 
Therefore, by condition~(U) of Section~\ref{SecConvAbel} the number of vertices 
in $V'_k$ is less than $\Const\cdot e\cdot (|Z-W|/h'+4k)$. Therefore, we can 
estimate the energy for a path $v_0v_1\dots v_N$ in $V_k'$ from $v_0=Z_k$, 
$v_N=W_k$ using~\eqref{eqestPath}
\begin{align*}
E_{Z_k\dots W_k}(u)\geq &\frac{\Const}{e\cdot (\Const\cdot e\cdot 
(|Z-W|/h'+4k))}\cdot (u(Z_k)-u(W_k))^2 \\ 
&= \Const\cdot \frac{(u(Z)-u(W))^2}{e^2}\cdot \frac{h'}{km'+ |Z-W|/4}.
\end{align*}
Summing these estimates and estimating $\sum_{k=1}^{m} 
\frac{h'}{kh'+ |Z-W|/4} \geq \Const \int_{h'}^{\frac{r-|Z-W|}{2}} \frac{dt}{t+ 
|Z-W|/4}$ we get
\begin{align*}
 E'(u)&\geq \sum_{k=1}^{m} E_{Z_k\dots W_k}(u) \geq 
\Const\cdot \frac{(u(Z)-u(W))^2}{e^2} \cdot 
\int_{h'}^{\frac{r-|Z-W|}{2}}\frac{dt}{t+ |Z-W|/4} \\
 &\geq \Const\cdot \frac{(u(Z)-u(W))^2}{e^2}\cdot \log 
\frac{2r-|Z-W|}{4h'+|Z-W|} \\
&\geq \Const\cdot \frac{(u(Z)-u(W))^2}{e^2}\cdot \log 
\frac{r}{3\max\{|Z-W|,h'\}}.
\end{align*}
This implies the desired inequalities~\eqref{eqequiu1} and~\eqref{eqequiu2}.


\bibliographystyle{amsalpha} 
\bibliography{ConvergencePeriodMatrices_Arxiv}

\providecommand{\bysame}{\leavevmode\hbox to3em{\hrulefill}\thinspace}
\providecommand{\MR}{\relax\ifhmode\unskip\space\fi MR }
\providecommand{\MRhref}[2]{%
  \href{http://www.ams.org/mathscinet-getitem?mr=#1}{#2}
}
\providecommand{\href}[2]{#2}
\begin{thebibliography}{MDSB03}

\bibitem[BG16]{BG16}
Alexander~I. Bobenko and Felix G\"{u}nther, \emph{Discrete complex analysis on
  planar quad-graphs}, Advances in discrete differential geometry, Springer,
  2016, pp.~57--132.

\bibitem[BG17]{BG17}
\bysame, \emph{Discrete {R}iemann surfaces based on quadrilateral cellular
  decompositions}, Adv. Math. \textbf{311} (2017), 885--932.

\bibitem[BMS05]{BMS05}
A.~I. Bobenko, Ch. Mercat, and Yu.~B. Suris, \emph{Linear and nonlinear
  theories of discrete analytic functions. {I}ntegrable structure and
  isomonodromic {G}reen's function}, J. reine angew. Math. \textbf{583} (2005),
  117--161.

\bibitem[BMS11]{BMS11}
A.~I. Bobenko, Ch. Mercat, and M.~Schmies, \emph{Period matrices of polyhedral
  surfaces}, Computational {A}pproach to {R}iemann {S}urfaces (A.~I. Bobenko
  and Ch. Klein, eds.), vol. 2013, Springer, 2011, pp.~213--226.

\bibitem[BP96]{BP96}
A.~I. Bobenko and U.~Pinkall, \emph{Discrete isothermic surfaces}, J. Reine
  Angew. Math. \textbf{475} (1996), 187--208.

\bibitem[BPS15]{BPS13}
A.I. Bobenko, U.~Pinkall, and B.~Springborn, \emph{Discrete conformal maps and
  ideal hyperbolic polyhedra}, Geom. Topol. \textbf{19} (2015), no.~4,
  2155--2215.

\bibitem[BS04]{BS02}
A.~I. Bobenko and B.~A. Springborn, \emph{Variational principles for circle
  patterns and {K}oebe's theorem}, Trans. Amer. Math. Soc. \textbf{356} (2004),
  659--689.

\bibitem[BS16]{BoSk16}
A.~I. Bobenko and M.~Skopenkov, \emph{Discrete {R}iemann surfaces: linear
  discretization and its convergence}, J. Reine Angew. Math. \textbf{720}
  (2016), 217--250.

\bibitem[B{\"u}c08]{Bue08}
U.~B{\"u}cking, \emph{Approximation of conformal mappings by circle patterns},
  Geom. Dedicata \textbf{137} (2008), 163--197.

\bibitem[B{\"u}c16]{Bue16}
\bysame, \emph{Approximation of conformal mappings on triangular lattices},
  Advances in Discrete Differential Geometry (A.I. Bobenko, ed.), Springer,
  2016.

\bibitem[CFL28]{CFL28}
R.~Courant, K.~Friedrichs, and H.~Lewy, \emph{{\"U}ber die partiellen
  {D}ifferentialgleichungen der mathematischen {P}hysik}, Math. Ann.
  \textbf{100} (1928), 32--74.

\bibitem[CS11]{ChS11}
D.~Chelkak and St. Smirnov, \emph{Discrete complex analysis on isoradial
  graphs}, Adv. in Math. \textbf{228} (2011), no.~3, 1590 -- 1630.

\bibitem[DN03]{DN03}
I.A. Dynnikov and S.P. Novikov, \emph{Geometry of the triangle equation on
  two-manifolds}, Mosc. Math. J. \textbf{3} (2003), 419--438.

\bibitem[Duf53]{Du53}
R.~J. Duffin, \emph{Discrete potential theory}, Duke Math. J. \textbf{20}
  (1953), 233--251.

\bibitem[Duf56]{Du56}
\bysame, \emph{Basic properties of discrete analytic functions}, Duke Math. J.
  \textbf{23} (1956), 335--363.

\bibitem[Duf59]{Du59}
\bysame, \emph{Distributed and lumped networks}, J. Math. Mech. \textbf{8}
  (1959), 793--826.

\bibitem[Duf68]{Du68}
\bysame, \emph{Potential theory on a rhombic lattice}, J. Combin. Th.
  \textbf{5} (1968), 258--272.

\bibitem[DvH01]{DH01}
B.~Deconinck and M.~van Hoeij, \emph{Computing {R}iemann matrices of algebraic
  curves}, Physica D \textbf{152-153} (2001), 28 -- 46, Advances in Nonlinear
  Mathematics and Science: A Special Issue to Honor Vladimir Zakharov.

\bibitem[Fer44]{F}
J.~Ferrand, \emph{Fonctions pr\'eharmoniques et fonctions pr\'eholomorphes},
  Bull. Sci. Math. \textbf{68} (1944), 152--180.

\bibitem[FK15]{FK15}
J.~Frauendiener and Ch. Klein, \emph{Computational approach to hyperelliptic
  {R}iemann surfaces}, Lett. Math. Phys. \textbf{105} (2015), no.~3, 379--400.

\bibitem[FK17]{FK17}
\bysame, \emph{Computational approach to compact {R}iemann surfaces},
  Nonlinearity \textbf{30} (2017), no.~1, 138--172.

\bibitem[GSST98]{GSST98}
P.~Gianni, M.~Seppälä, R.~Silhol, and B.~Trager, \emph{Riemann surfaces,
  plane algebraic curves and their period matrices}, J. Symbolic Comput.
  \textbf{26} (1998), no.~6, 789 -- 803.

\bibitem[Isa41]{Is41}
R.~Ph. Isaacs, \emph{A finite difference function theory}, Univ. Nac.
  Tucum{\'a}n Revista A \textbf{2} (1941), 177--201.

\bibitem[LF55]{LF55}
J.~Lelong-Ferrand, \emph{Repr{\'e}sentation confrome et transformations {\`a}
  int{\'e}grale de {D}irichlet born{\'e}e}, Gauthier-Villars, Paris, 1955.

\bibitem[Lus26]{Lu26}
L.~Lusternik, \emph{{\"U}ber einige {A}nwendunge der direkten {M}ethoden in der
  {V}ariationsrechnung}, Mat. Sb. \textbf{33} (1926), no.~2, 173--201.

\bibitem[Mac49]{MacN49}
R.~MacNeal, \emph{The solution of partial differential equations by means of
  electrical networks}, Ph.D. thesis, California Institute of Technology, 1949.

\bibitem[Mat05]{Ma05}
D.~Matthes, \emph{Convergence in discrete {C}auchy problems and applications to
  circle patterns}, Conform. Geom. Dyn. \textbf{9} (2005), 1--23.

\bibitem[MDSB03]{MDSB03}
M.~Meyer, M.~Desbrun, P.~Schr\"oder, and A.~H. Barr, \emph{Discrete
  differential geometry operators for triangulated 2-manifolds}, Visualization
  and Mathematics III, Springer, 2003, pp.~35--57.

\bibitem[Mer01]{Me01}
Ch. Mercat, \emph{Discrete {R}iemann surfaces and the {I}sing model}, Commun.
  Math. Phys. \textbf{218} (2001), 177--216.

\bibitem[Mer02]{Me02}
\bysame, \emph{Discrete period matrices and related topics}, e-print
  arXiv:math-ph/0111043, 2002.

\bibitem[Mer07]{Me07}
\bysame, \emph{Discrete {R}iemann surfaces}, Handbook of Teichmüller Theory
  (A.~Papadopoulos, ed.), IRMA Lectures in Mathematics and Theoretical Physics,
  vol.~11, Eur. Math. Soc., 2007, pp.~541--575.

\bibitem[MN17]{MN17}
P.~{Molin} and C.~{Neurohr}, \emph{Computing period matrices and the
  {A}bel-{J}acobi map of superelliptic curves}, e-print arXiv:1707.07249
  [math.NT], 2017.

\bibitem[Nov11]{N11}
S.P. Novikov, \emph{New discretization of complex analysis: the {E}uclidean and
  {H}yperbolic planes}, Proceedings of the Steklov Institute of Mathematics
  \textbf{273} (2011), 238--251.

\bibitem[PP93]{PP93}
U.~Pinkall and K.~Polthier, \emph{Computing discrete minimal surfaces and their
  conjugates}, Experiment. Math. \textbf{2} (1993), 15--36.

\bibitem[Sch97]{Sch97}
O.~Schramm, \emph{Circle patterns with the combinatorics of the square grid},
  Duke Math. J. \textbf{86} (1997), 347--389.

\bibitem[Sko13]{Sko13}
M.~Skopenkov, \emph{The boundary value problem for discrete analytic
  functions}, Adv. Math. \textbf{240} (2013), 61--87.

\bibitem[Ste05]{St05}
K.~Stephenson, \emph{Introduction to circle packing: the theory of discrete
  analytic functions}, Cambridge University Press, New York, 2005.

\end{thebibliography}

\end{document}